\documentclass[11pt]{article}
\usepackage{epigamath}

\usepackage[notext]{kpfonts}
\usepackage{baskervald}

\setpapertype{A4}

\usepackage[english]{babel}

\usepackage[dvips]{graphicx}     

\title{\vspace{-1.5cm}The parabolic exotic t-structure}
\titlemark{The parabolic exotic t-structure}
\author{\vspace{0cm} Pramod N. Achar, Nicholas Cooney and Simon Riche}
\authoraddresses{
\authordata{Pramod N. Achar}{\firstname{Pramod N.} \lastname{Achar}\\
\institution{Department of Mathematics,
  Louisiana State University,
  Baton Rouge, LA 70803,
  U.S.A.}\\
\email{pramod@math.lsu.edu}}\\
\authordata{Nicholas Cooney}{\firstname{Nicholas}
\lastname{Cooney}\\
\institution{Universit\'e Clermont Auvergne, CNRS, LMBP, F-63000 Clermont-Ferrand, France \\ Max Planck Institute for Mathematics, Bonn, Germany (present affiliation)}\\
\email{ncooney@mpim-bonn.mpg.de}}\\
\authordata{Simon Riche}{\firstname{Simon}
\lastname{Riche}\\
\institution{Universit\'e Clermont Auvergne, CNRS, LMBP, F-63000 Clermont-Ferrand, France}\\
\email{simon.riche@uca.fr}}
}
\authormark{P. A. Achar, N. Cooney, and S. Riche}
\date{\vspace{-5ex}} 
\journal{\'Epijournal de G\'eom\'etrie Alg\'ebrique} 
\acceptation{Received by the Editors on May 22, 2018.\\
Accepted on October 4, 2018.}

\newcommand{\articlereference}{Volume 2 (2018), Article Nr. 8}

\acknowledgement{P.A. was supported by NSF Grant No.~DMS-1500890. This project has received funding from the European Research Council (ERC) under the European Union's Horizon 2020 research and innovation programme (grant agreement No 677147).}

 \usepackage[all]{xy}

\allowdisplaybreaks

\numberwithin{equation}{numsection}
\renewcommand{\theequation}{\arabic{section}.\arabic{equation}}

\usepackage{tikz}
\usetikzlibrary{cd}

\newtheorem{thm}{Theorem}[section]
\newtheorem{lem}[thm]{Lemma}

\newtheorem{cor}[thm]{Corollary}

\newtheorem{rmk}[thm]{Remark}

\numberwithin{figure}{section}
\numberwithin{table}{section}

\makeatletter
\let\c@table\c@figure
\makeatother

\newtheorem{lemA}{Lemma}
\newtheorem{propA}[lemA]{Proposition}
\newtheorem{rmkA}[lemA]{Proposition}

\newcommand{\bk}{\Bbbk}
\newcommand{\Z}{\mathbb{Z}}
\newcommand{\C}{\mathbb{C}}

\newcommand{\Ql}{\mathbb{Q}_\ell}
\newcommand{\F}{\mathbb{F}}

\newcommand{\Gm}{\mathbb{G}_{\mathrm{m}}}

\newcommand{\fn}{\mathfrak{n}}

\newcommand{\fR}{\mathfrak{R}}

\newcommand{\weyl}{\mathsf{M}}
\newcommand{\coweyl}{\mathsf{N}}
\newcommand{\tilt}{\mathsf{T}}
\newcommand{\irr}{\mathsf{L}}

\newcommand{\tcN}{{\widetilde{\mathcal{N}}}}
\newcommand{\se}{\mathsf{e}}

\newcommand{\bX}{\mathbf{X}}

\newcommand{\bXpp}{\bX^{+,\mathrm{reg}}}

\newcommand{\Waff}{W_{\mathrm{aff}}}
\newcommand{\WaffCox}{W_{\mathrm{aff}}^{\mathrm{Cox}}}
\newcommand{\Waffmin}{{}^0 \hspace{-1pt} \Waff}
\newcommand{\WaffminCox}{{}^0 \hspace{-1pt} \WaffCox}
\newcommand{\WaffminI}{{}^0 \hspace{-1pt} \Waff^I}

\newcommand{\cF}{\mathcal{F}}
\newcommand{\cE}{\mathcal{E}}
\newcommand{\cG}{\mathcal{G}}
\newcommand{\cH}{\mathcal{H}}

\newcommand{\cO}{\mathcal{O}}
\newcommand{\cP}{\mathcal{P}}

\newcommand{\bQ}{\mathbf{Q}}
\newcommand{\cJ}{\mathcal{J}}
\newcommand{\cT}{\mathcal{T}}

\newcommand{\Gr}{\mathrm{Gr}}
\newcommand{\Fl}{\mathrm{Fl}}
\newcommand{\sph}{{\mathrm{sph}}}
\newcommand{\Sat}{\mathcal{S}}
\newcommand{\Iw}{\mathrm{Iw}}

\newcommand{\mix}{{\mathrm{mix}}}
\newcommand{\Perv}{\mathsf{Perv}}
\newcommand{\Par}{\mathsf{Parity}}

\newcommand{\Rep}{\mathsf{Rep}}

\newcommand{\Coh}{\mathsf{Coh}}

\newcommand{\ExCoh}{\mathsf{ExCoh}}
\newcommand{\grRep}{\mathsf{grRep}}

\newcommand{\Db}{D^{\mathrm{b}}}
\newcommand{\Kb}{K^{\mathrm{b}}}

\newcommand{\Dmix}{D^\mix}

\DeclareMathOperator{\Hom}{Hom}
\DeclareMathOperator{\Ext}{Ext}
\DeclareMathOperator{\End}{End}

\DeclareMathOperator{\For}{For}

\newcommand{\id}{\mathrm{id}}
\newcommand{\simto}{\xrightarrow{\sim}}
\newcommand{\la}{\langle}
\newcommand{\ra}{\rangle}

\newcommand{\TilRT}{\mathsf{T}^{\mathrm{RT}}}
\newcommand{\IW}{\mathcal{IW}}
\newcommand{\Av}{\mathsf{Av}}
\newcommand{\Whit}{\mathrm{Wh}}
\newcommand{\Tilt}{\mathsf{Tilt}}

\newcommand{\GKM}{\mathscr{G}}
\newcommand{\BKM}{\mathscr{B}}
\newcommand{\PKM}{\mathscr{P}}

\newcommand{\UKM}{\mathscr{U}}
\newcommand{\LKM}{\mathscr{L}}
\newcommand{\Flag}{\mathscr{X}}
\newcommand{\AS}{\mathcal{L}_{\mathrm{AS}}}
\newcommand{\KW}{{}^K \hspace{-1pt} W}
\newcommand{\KD}{{}^K \hspace{-1pt} \Delta}
\newcommand{\KN}{{}^K \hspace{-1pt} \nabla}

\makeatletter
\def\lotimes{\@ifnextchar_{\@lotimessub}{\@lotimesnosub}}
\def\@lotimessub_#1{\mathchoice{\mathbin{\mathop{\otimes}^L}_{#1}}%
  {\otimes^L_{#1}}{\otimes^L_{#1}}{\otimes^L_{#1}}}
\def\@lotimesnosub{\mathbin{\mathop{\otimes}^L}}
\makeatother

\newcommand{\sslash}{\,\,\slash\!\!\slash\,\,}

\begin{document}


\maketitle



\begin{prelims}

\vspace{-0.45cm}

\def\abstractname{Abstract}
\abstract{Let $G$ be a connected reductive algebraic group over an algebraically closed field $\Bbbk$, with simply connected derived subgroup. The \emph{exotic t-structure} on the cotangent bundle of its flag variety $T^*(G/B)$, originally introduced by Bezrukavnikov, has been a key tool for a number of major results in geometric representation theory, including the proof of the graded Finkelberg--Mirkovi\'c conjecture.
In this paper, we study (under mild technical assumptions) an analogous t-structure on the cotangent bundle of a partial flag variety $T^*(G/P)$.  As an application, we prove a parabolic analogue of the Arkhipov--Bezrukavnikov--Ginzburg equivalence.  When the characteristic of $\bk$ is larger than the Coxeter number, we deduce an analogue of the graded Finkelberg--Mirkovi\'c conjecture for some singular blocks.}

\keywords{Flag varieties; derived category of coherent sheaves; parity complexes; t-structure; exceptional collection; highest weight categories}

\MSCclass{14M15 (primary); 20G05; 20G44 (secondary)}


\languagesection{Fran\c{c}ais}{%

\vspace{-0.15cm}
\textbf{Titre. La t-structure exotique parabolique.} Soit $G$ un groupe alg\'ebrique r\'eductif connexe sur un corps $\Bbbk$ alg\'ebriquement clos. La \emph{t-structure exotique} sur le fibr\'e cotangent de sa vari\'et\'e de drapeaux $T^*(G/B)$, introduite \`a l'origine par Bezrukavnikov, a \'et\'e un outil cl\'e pour de nombreux r\'esultats majeurs en th\'eorie g\'eom\'etrique des repr\'esentations, en particulier
la d\'emonstration de la conjecture de Finkelberg--Mirkovi\'c gradu\'ee.
Dans cet article, nous \'etudions (sous de l\'eg\`eres hypoth\`eses techniques) une t-structure analogue sur le fibr\'e cotangent de la vari\'et\'e de drapeaux partiels $T^*(G/P)$. Comme application, nous prouvons un analogue parabolique de l'\'equivalence de Arkhipov--Bezrukavnikov--Ginzburg. Lorsque la caract\'eristique de $\Bbbk$ est sup\'erieure au nombre de Coxeter, nous d\'eduisons un analogue de la conjecture de Finkelberg--Mirkovi\'c gradu\'ee pour certains blocs singuliers.}

\end{prelims}


\newpage

\setcounter{tocdepth}{1} \tableofcontents

\section{Introduction}
\label{s:intro}

\subsection{The exotic t-structure}

Let $\dot G$ be a connected reductive algebraic group over an algebraically closed field $\bk$, and let $\dot B \subset \dot G$ be a Borel subgroup.  (The undecorated letter $G$ is reserved for another group, to be introduced later.)   Let $\tcN = T^*(\dot G/\dot B)$ be the cotangent bundle of its flag variety, and consider the derived category $\Db \Coh^{\dot G \times \Gm}(\tcN)$ of $(\dot G \times \Gm)$-equivariant coherent sheaves on $\tcN$. The \emph{exotic t-structure} is a remarkable t-structure on this category, originally defined in~\cite{bez:ctm}.  This t-structure has close connections to derived equivalences coming from the geometric Langlands program~\cite{abg:qglg, ar:agsr, mr:etsps}, to the cohomology of tilting modules for Lusztig's quantum groups~\cite{bez:ctm} and algebraic groups~\cite{ahr}, and other topics in representation theory (see~\cite{achar} and the references therein). It is defined using a so-called \emph{exceptional set} of objects in $\Db\Coh^{\dot G \times \Gm}(\tcN)$.

An important and rather nontrivial feature of this t-structure is that \emph{the higher t-cohomology of every exceptional object vanishes}.  This theorem (which was implicit in~\cite{bez:ctm} and proved in different ways in~\cite{ar:agsr, mr:etspc}) implies that the heart of this t-structure has the familiar structure of a \emph{highest weight category}.  For representation-theoretic applications, this fact plays a similar conceptual role to the Kempf vanishing theorem (for reductive groups) or to the Artin vanishing theorem (for direct images of perverse sheaves under affine maps). In particular, this vanishing theorem plays a crucial role in the proof of the graded Finkelberg--Mirkovi\'c conjecture~\cite{prinblock}, which relates the principal block of a reductive group to perverse sheaves on the Langlands dual affine Grassmannian, hence also in the proof of the tilting character formula for reductive groups~\cite{amrw}.

\subsection{Parabolic analogue}

The main result of this paper is a parabolic version of this vanishing theorem. Namely, for any parabolic subgroup $\dot P \subset \dot G$, the first and third authors have defined in~\cite{prinblock} a certain exceptional set in the triangulated category $\Db \Coh^{\dot G \times \Gm}(T^*(\dot G/\dot P))$ which generalizes that defined by Bezrukavnikov in the case $\dot P= \dot B$. As in this special case, our exceptional set determines a t-structure on $\Db \Coh^{\dot G \times \Gm}(T^*(\dot G/\dot P))$, which we again call the exotic t-structure.  In Theorem~\ref{thm:main} we show that, with respect to this t-structure, the higher t-cohomology of every exceptional object vanishes. As a consequence, its heart is a highest weight category; see Corollary~\ref{cor:hw}.

Unlike the case of the full flag variety as treated in~\cite{mr:etspc}, our proof is indirect, and requires a translation of the problem to the realm of constructible sheaves and ``mixed derived categories'' in the sense of~\cite{modrap2}.

\subsection{Applications}

As an application, we prove a modular parabolic version of a derived equivalence originally due to Arkhipov--Bezrukavnikov--Ginzburg, relating $\Db \Coh^{\dot G \times \Gm}(T^*(\dot G/\dot P))$ to a category of constructible sheaves on the affine Grassmannian of the Langlands dual group.

Then, under the additional assumption that the characteristic of $\bk$ is larger than the Coxeter number of $\dot G$, we combine this result with~\cite{prinblock} to deduce a singular version of the graded Finkelberg--Mirkovi\'c conjecture, relating a certain block (whose ``singularity'' is controlled by $\dot P$) for the reductive group whose Frobenius twist is $\dot G$ to a suitable category of Whittaker perverse sheaves on the dual affine Grassmannian.

\subsection{Contents}

In Section~\ref{sec:statement} we state our main result more precisely, and outline our strategy of proof. This result is proved in Section~\ref{s:mainproof}, after some preliminaries in Section~\ref{s:reptheory}.
The applications are deduced in Section~\ref{s:singular-fm}. Finally, in Appendix~\ref{sec:appendix} we extend certain results on parity complexes and ``mixed derived categories'' to the case of Whittaker sheaves. (These results play a technical role in some of our proofs.)

\subsection*{Acknowledgements}

We thank G.~Williamson for useful discussions.

\section{Statement of the main result}
\label{sec:statement}

\subsection{Notation}
\label{ss:notation}

Let $\dot G$ be a connected reductive algebraic group over a field $\bk$ of characteristic $\ell$, with maximal torus and Borel subgroup $\dot T\subset \dot B\subset \dot G$. (The reason why we decorate our notation with a dot should become clear in~\S\ref{ss:relation} below. It does not play any role in earlier subsections.) Let also $\dot B^+$ be the Borel subgroup which is opposite to $\dot B$ (with respect to $\dot T$),
$\fR$ be the root system of $(\dot G,\dot T)$, and $\fR^+\subset\fR$ be the system of positive roots given by the nonzero $\dot T$-weights in $\mathrm{Lie}(\dot B^+)$. We will denote by $\bX$ the character lattice of $\dot T$, and by $S$ the set of simple reflections of the Weyl group $W$ of $(\dot G,\dot T)$ determined by our choice of $\fR^+$. For $s \in S$, we will denote by $\alpha_s$ the corresponding simple root, and by $\alpha_s^\vee$ the associated coroot.

We will make the following assumptions on $\dot G$ and $\bk$:
\begin{enumerate}
 \item[\rm (1)]
 \phantomsection\label{it:assumption-verygood}
 $\ell$ is very good for $\dot G$;
 \item[\rm (2)]
 \phantomsection\label{it:assumption-Gder}
 the derived subgroup of $\dot G$ is simply connected.
 \end{enumerate}
By~\cite[Proposition~2.5.12]{letellier}, our assumption~(\hyperref[it:assumption-verygood]{1}) implies the following property:
 \begin{equation} 
 \label{it:assumption-bilinear-form}
 \text{$\mathrm{Lie}(\dot G)$ admits a nondegenerate $\dot G$-invariant bilinear form.}
 \end{equation}
These assumptions and this property will in particular allow us to use the results of~\cite{ar:agsr}.
(Note that, by~\cite[Corollary~1.6]{mr:etsps}, the condition that $\ell$ is a ``JMW prime'' used in~\cite{ar:agsr} is equivalent to the condition that $\ell$ is good for $\dot G$.)

Our assumption~(\hyperref[it:assumption-Gder]{2}) allows us to choose weights $(\varsigma_s)_{s \in S}$ such that
\[
 \langle \varsigma_s, \alpha_t^\vee \rangle = \delta_{s,t}
\]
for any $s,t \in S$. For any subset $I \subset S$, we then set $\varsigma_I:=\sum_{s \in I} \varsigma_s$.

The main players of this article will be the ``partial Springer resolutions''
\[
 \tcN_{I}:=\dot G\times^{\dot P_I}\dot{\fn}_I
\]
for $I \subset S$, where $\dot P_I \subset \dot G$ is the standard (with respect to $\dot B$) parabolic subgroup corresponding to $I$, and $\dot{\fn}_I$ is the Lie algebra of its unipotent radical. Note that $\tcN_\varnothing$ is the usual Springer resolution of $\dot G$, and that $\tcN_I$ identifies 
with the cotangent bundle to $\dot G / \dot P_I$ (thanks to~\eqref{it:assumption-bilinear-form}).

We let the multiplicative group $\Gm$ act on $\dot{\fn}_I$ by $z \cdot x = z^{-2}x$.  This induces an action on $\tcN_I$ that commutes with the natural $\dot G$-action, so we may consider $\dot G \times \Gm$-equivariant coherent sheaves. We will denote by 
\[
\langle 1 \rangle: \Coh^{\dot G \times \Gm}(\tcN_I)\to \Coh^{\dot G \times \Gm}(\tcN_I)
\] 
the functor of tensoring with the tautological $1$-dimensional $\Gm$-module, and by $\langle n \rangle$ the $n^\mathrm{th}$-power of this functor (for $n \in \Z$). 

\subsection{A graded exceptional set}

Following the notation and conventions of~\cite[\S 9]{prinblock}, we fix a subset $I \subset S$, and consider the objects $\Delta_I(\lambda)$ and $\nabla_I(\lambda)$ in $\Db \Coh^{\dot G \times \Gm}(\tcN_I)$ characterized in~\cite[Proposition~9.16]{prinblock}.\footnote{In~\cite{prinblock} we work under the running assumption that $\ell$ is bigger than the Coxeter number of $\dot G$. However, as noticed in~\cite[Remark~9.1]{prinblock}, the results of this particular section hold in the present generality.} Here $\lambda\in\bXpp_I$ where
\[
 \bXpp_I:=\{\lambda \in \bX \mid \forall s \in I, \, \langle \lambda, \alpha_s^\vee \rangle >0\}.
\]
In order to define these objects one needs to choose an order $\leq'$ on $\bX$. Here we will assume that $\leq'$ is constructed as in~\cite[\S9.4]{prinblock}. Then the objects one obtains are independent of the choices involved in this construction, by~\cite[Proposition~9.19(1) and Proposition~9.24]{prinblock}.

\begin{rmk}{\rm
 In the case $I=\varnothing$, the objects $\Delta_I(\lambda)$ and $\nabla_I(\lambda)$ are the same (up to shift) as those introduced by Bezrukavnikov (for characteristic-$0$ coefficients) in~\cite{bez:ctm}. The general case is similar, replacing characters of $\dot T$ by standard or costandard modules for the Levi factor of $\dot P_I$ containing $\dot T$. In particular, when $I=S$, the object $\Delta_S(\lambda)$ is the Weyl module of highest weight $\lambda-\varsigma_S$, and $\nabla_S(\lambda)$ is the induced module of highest weight $\lambda-\varsigma_S$.}
\end{rmk}

According to~\cite[Proposition~9.16]{prinblock}, the objects $(\nabla_I(\lambda): \lambda \in \bXpp_I)$ form a graded exceptional set of objects with respect to the order $\leq'$ and the ``shift functor'' $\langle 1 \rangle$, in the sense of~\cite[\S 2.1.5]{bez:ctm}. That is, we have
\begin{equation}\label{eqn:excep1}
\Hom_{\Db \Coh^{\dot G \times \Gm}(\tcN_I)}(\nabla_I(\lambda),\nabla_I(\mu)\langle n \rangle[m])=0
\end{equation}
 if $\mu \not\leq' \lambda$ or if $\lambda=\mu$ and $(n,m)\neq(0,0)$, and moreover
\begin{equation}\label{eqn:excep2}
 \Hom_{\Db \Coh^{\dot G \times \Gm}(\tcN_I)}(\nabla_{I}(\lambda),\nabla_{I}(\lambda))=\bk.
\end{equation}
The dual exceptional set is given by $(\Delta_I(\lambda): \lambda \in \bXpp_I)$. In other words, these objects form the unique collection of objects satisfying 
\[
\Hom_{\Db \Coh^{\dot G \times \Gm}(\tcN_I)}(\Delta_I(\lambda),\nabla_I(\mu)\langle n \rangle [m])=0
\]
 if $\mu <' \lambda$ and 
 \begin{equation}\label{eqn:dual-excep}
 \Delta_I(\lambda)\cong\nabla_I(\lambda)\mod\Db \Coh^{\dot G \times \Gm}(\tcN_I)_{<' \lambda}.
 \end{equation}
 Here $\Db \Coh^{\dot G \times \Gm}(\tcN_I)_{<' \lambda}$ is the full triangulated subcategory of $\Db \Coh^{\dot G \times \Gm}(\tcN_I)$ generated by the objects $\nabla_I(\mu)\langle n \rangle$ for $\mu<'\lambda$ and $n\in\Z$, and the condition~\eqref{eqn:dual-excep} means that the images of $\Delta_I(\lambda)$ and $\nabla_I(\lambda)$ in the Verdier quotient
 \[
 \Db \Coh^{\dot G \times \Gm}(\tcN_I)/\Db \Coh^{\dot G \times \Gm}(\tcN_I)_{<' \lambda}
 \]
 are isomorphic. 
These objects in fact satisfy
 \begin{equation}
 \label{eqn:hom-exc}
  \Hom_{\Db \Coh^{\dot G \times \Gm}(\tcN_I)}(\Delta_I(\mu), \nabla_I(\nu) \langle n \rangle [m]) \cong \begin{cases}                                                                                                        \bk & \text{if $\mu=\nu$ and $n=m=0$;}\\
  0 & \text{otherwise,}                                                                                  \end{cases}
 \end{equation}
see~\cite[Corollary~9.18]{prinblock}. Moreover, both of the families $(\nabla_I(\lambda) \langle n \rangle: \lambda \in \bXpp_I, \, n \in \Z)$ and $(\Delta_I(\lambda) \langle n \rangle: \lambda \in \bXpp_I, \, n \in \Z)$ generate $\Db \Coh^{\dot G \times \Gm}(\tcN_I)$ as a triangulated category.

\subsection{The exotic t-structure}
\label{ss:main}

By the general theory of (graded) exceptional sets (see~\cite[Proposition~4]{bez:ctm}), the following pair of subcategories defines a bounded t-structure 
on the triangulated category $\Db \Coh^{\dot G \times \Gm}(\tcN_I)$:
\begin{align*}
 \Db \Coh^{\dot G \times \Gm}(\tcN_I)^{\leq 0} &= \langle \Delta_I(\lambda) \langle n \rangle [m] : \lambda \in \bXpp_I, \, n \in \Z, \, m \in \Z_{\geq 0} \rangle_{\mathrm{ext}};\\
 \Db \Coh^{\dot G \times \Gm}(\tcN_I)^{\geq 0} &= \langle \nabla_I(\lambda) \langle n \rangle [m] : \lambda \in \bXpp_I, \, n \in \Z, \, m \in \Z_{\leq 0} \rangle_{\mathrm{ext}}.
\end{align*}
Here, $\langle \mathcal{A} \rangle_{\mathrm{ext}}$ means the smallest strictly full additive subcategory containing the objects $\mathcal{A}$ and closed under extensions. This t-structure is called the \emph{exotic t-structure}, and its heart will be denoted by
\[
 \ExCoh(\tcN_I).
\]
It is clear from the definitions that the functor $\langle 1 \rangle$ is t-exact for this t-structure.

The main result of this paper is the following.

\begin{thm}
\label{thm:main}
 The objects $\Delta_I(\lambda)$ and $\nabla_I(\lambda)$ ($\lambda \in \bXpp_I$) belong to $\ExCoh(\tcN_I)$.
\end{thm}

This result can be rephrased as a cohomology-vanishing statement as follows: since $\nabla_I(\lambda)$ belongs to $\Db \Coh^{\dot G \times \Gm}(\tcN_I)^{\geq 0}$ by definition, Theorem~\ref{thm:main} is equivalent to the statement that
\begin{equation}
\label{eqn:cohom-vanishing}
{}^{\mathrm{t}}\mathrm{H}^i(\nabla_I(\lambda)) = 0 \qquad\text{for $i > 0$,}
\end{equation}
along with a similar vanishing statement in negative degrees for $\Delta_I(\lambda)$ (where ${}^{\mathrm{t}}\mathrm{H}^i$ means the $i$-th cohomology with respect to the exotic t-structure).

Once Theorem~\ref{thm:main} is established,
standard arguments (see e.g.~\cite{bgs},~\cite[\S 3.5]{mr:etspc} or~\cite[Proposition~3.11]{modrap2}) then imply the following claim, which formed our main motivation for studying this question.

\begin{cor}
\label{cor:hw}
 The category $\ExCoh(\tcN_I)$ is a graded highest weight category in the sense of~\cite[\S 7]{riche-hab}, with weight poset $(\bXpp_I, \leq')$, standard objects $(\Delta_I(\lambda): \lambda \in \bXpp_I)$, costandard objects $(\nabla_I(\lambda): \lambda \in \bXpp_I)$, and shift functor~$\langle 1 \rangle$.
\end{cor}

\subsection{Strategy of proof}
\label{ss:strategy}

The proof of Theorem~\ref{thm:main} will be given in Section~\ref{s:mainproof}. In the case $I=\varnothing$, this theorem is one of the main results of~\cite{mr:etsps} (see also~\cite{ar:agsr} for a different proof). This special case plays a crucial role in the proof for $I \neq \varnothing$.

Broadly speaking, the strategy of the proof is to carry out a kind of categorical ``diagram chase'' using the categories and functors on the second, third and fourth columns in the diagram of Figure~\ref{fig:main}.  (Precise definitions of all the notation in this diagram will be given in the following sections; we only mention here that the categories in the third and fourth columns are certain ``mixed derived categories'' in the sense of~\cite{modrap2}, and that the dashed arrow is not an equivalence but an identification of the right-hand side with a certain summand in the left-hand side; see~\S\ref{ss:proof-adv} for details.)  The leftmost column of this figure is only defined under the stronger assumption that $\ell$ is larger than the Coxeter number of $\dot G$, but it motivates our constructions even when it is not available.  Each category carries one or two t-structures, which in some cases are already known to satisfy analogues of~\eqref{eqn:cohom-vanishing}.  Table~\ref{tab:tstruc} lists the t-structures that will come up in this paper.  In this table, t-structures appearing in the same row correspond to one another under one of the horizontal functors in Figure~\ref{fig:main}. 

In more detail, we begin in Section~\ref{s:reptheory} by defining and studying a second t-structure on $\Db\Coh^{\dot G \times \Gm}(\tcN_\varnothing)$, and relating both t-structures to the affine Grassmannian $\Gr'$ of the Langlands dual group.  In Section~\ref{s:mainproof}, we transfer the problem to the rightmost column of Figure~\ref{fig:main}.  Specifically, we will reduce the proof of~\eqref{eqn:cohom-vanishing} to a similar claim for the perverse t-structure on the category $\Dmix_{\IW}(\Fl_I,\bk)$. This claim is proved in Appendix~\ref{sec:appendix}, using variations on some arguments in~\cite{rw, amrw, modrap2}. (In a sense, the problem is easier for this category because one can use the categories of Bruhat-constructible sheaves on $\Fl$ and $\Fl_I$, which have no counterparts in the world of coherent sheaves.)

\begin{figure}
\[
\begin{tikzcd}
\Db\Rep_\varnothing(G) \ar[d, shift left, "T_\varnothing^I"] &
\Db \Coh^{\dot G \times \Gm}(\tcN_\varnothing) \ar[r, "\Psi", "\sim"'] \ar[l, "\Phi_\varnothing"'] \ar[d, shift left, "\Pi_I"] &
  \Dmix_{(\Iw)}(\Gr',\bk) \ar[r, dashrightarrow, "\kappa"] &
  \Dmix_{\IW}(\Fl,\bk) \ar[d, shift left, "(q_I)_*"] & 
  \\
\Db\Rep_I(G) \ar[u, shift left, "T_I^\varnothing"] &
\Db \Coh^{\dot G \times \Gm}(\tcN_I) \ar[l, "\Phi_I"'] \ar[u, shift left, "\Pi^I"]& &
  \Dmix_{\IW}(\Fl_I,\bk) \ar[u, shift left, "(q_I)^*"] 
\end{tikzcd}
\]
\caption{Setting for the proof of Theorem~\ref{thm:main}}\phantomsection\label{fig:main}
\end{figure}

\begin{table}
\bigskip\bigskip
\begin{center}
\begin{tabular}{c@{\quad}c@{\quad}c@{\quad}c}
$\Db\Rep_I(G)$ &
$\Db \Coh^{\dot G \times \Gm}(\tcN_I)$ &
$\Dmix_{(\Iw)}(\Gr', \bk)$ &
$\Dmix_{\IW}(\Fl_I,\bk)$ \\
\hline
& exotic & adverse & perverse \\
natural & representation-theoretic & perverse
\end{tabular}
\end{center}
\caption{t-structures arising in the proof}\phantomsection\label{tab:tstruc}
\end{table}

Once the proof of Theorem~\ref{thm:main} is complete, we will be in a position 
to obtain an analogue on the bottom line of Figure~\ref{fig:main} of the equivalence ``$\Psi$'' of the upper line.
We do this in Section~\ref{s:singular-fm}, and thereby obtain the parabolic version of the Arkhipov--Bezrukavnikov--Ginzburg equivalence (Theorem~\ref{thm:parabolic-abg}).
(The space $\Gr'$ appearing in Figure~\ref{fig:main} is the ``right coset'' version of the affine Grassmannian, but in Section~\ref{s:singular-fm} we will switch to the traditional ``left coset'' version, denoted by $\Gr$.)
Finally, when $\ell$ is larger than the Coxeter number, we can combine this equivalence with the results of~\cite{prinblock} to obtain the singular version of the graded Finkelberg--Mirkovi\'c conjecture (Theorem~\ref{thm:singular-fm}).

\section{Representation-theoretic t-structure and translation functors}
\label{s:reptheory}

\subsection{The representation-theoretic t-structure}

In this subsection, we will introduce and study a different t-structure on $\Db \Coh^{\dot G \times \Gm}(\tcN_I)$, which is ``Koszul dual'' (in an appropriate sense) to the exotic t-structure.

\begin{lem}
 The objects $(\nabla_I(\lambda) : \lambda \in \bXpp_I)$ form a graded exceptional set of objects in $\Db \Coh^{\dot G \times \Gm}(\tcN_I)$ with respect to the order $\leq'$ and the shift functor $\langle 1 \rangle [1]$.
\end{lem}

\begin{proof}
The claim follows from the observation that the definition of a graded exceptional set does not depend on the choice of ``shift functor.'' More precisely, in our case
the assertion we must prove states that
 \[ \Hom_{\Db \Coh^{\dot G \times \Gm}(\tcN_I)}(\nabla_{I}(\lambda),\nabla_{I}(\mu)\langle n \rangle [m+n])=0 
 \]
 if $\mu \not\leq' \lambda$ or if $\lambda=\mu$ and $(n,m)\neq(0,0)$, and moreover that
 \[
   \Hom_{\Db \Coh^{\dot G \times \Gm}(\tcN_I)}(\nabla_{I}(\lambda),\nabla_{I}(\lambda))=\bk.
 \]
But these are clearly equivalent to~\eqref{eqn:excep1} and~\eqref{eqn:excep2}.
\qed
\end{proof}

Of course the dual exceptional set is again $(\Delta_I(\lambda) : \lambda \in \bXpp_I)$. Using once again~\cite[Proposition~4]{bez:ctm}, we obtain that 
the following pair of subcategories forms a bounded t-structure
on $\Db \Coh^{\dot G \times \Gm}(\tcN_I)$:
\begin{align*}
 \Db \Coh^{\dot G \times \Gm}(\tcN_I)_{\mathrm{RT}}^{\leq 0} &= \langle \Delta_I(\lambda) \langle n \rangle [n+m] : \lambda \in \bXpp_I, \, n \in \Z, \, m \in \Z_{\geq 0} \rangle_{\mathrm{ext}};\\
 \Db \Coh^{\dot G \times \Gm}(\tcN_I)^{\geq 0}_{\mathrm{RT}} &= \langle \nabla_I(\lambda) \langle n \rangle [n+m] : \lambda \in \bXpp_I, \, n \in \Z, \, m \in \Z_{\leq 0} \rangle_{\mathrm{ext}}.
\end{align*}
This t-structure will be called the \emph{representation-theoretic} t-structure, and its heart will be denoted
\[
 \grRep(\tcN_I).
\]
By construction, the functor $\langle 1 \rangle[1]$ is t-exact with respect to this t-structure.

\begin{rmk}{\rm
The motivation for our terminology and notation should become clear in~\S\ref{ss:relation} below.}
\end{rmk}

\subsection{Geometric translation functors}
\label{ss:geometric-translation}

Now we will make use of the ``translation functors''
\[
\begin{tikzcd}[column sep=large]
\Db \Coh^{\dot G \times \Gm}(\tcN_\varnothing) \ar[r, bend left=5, "\Pi_I"] &
\Db \Coh^{\dot G \times \Gm}(\tcN_I) \ar[l, bend left=5, "\Pi^I"]
\end{tikzcd}
\]
defined in~\cite[\S 9.2]{prinblock} as follows.
Set $\tcN_{\varnothing,I}:=\dot G\times^{\dot B}\dot\fn_I$ and $n_I:=\dim(\dot P_I\slash\dot B)$. The inclusion $\dot\fn_I\hookrightarrow\dot\fn_\varnothing$ induces a $\dot G$-equivariant morphism $$\se_I:\tcN_{\varnothing,I}\hookrightarrow\tcN_\varnothing.$$ There is also a smooth proper map $$\mu_I:\tcN_{\varnothing,I}\rightarrow\tcN_I$$ with fibers isomorphic to $\dot P_I\slash\dot B$. We define:
\begin{align*}
\Pi_I(\cF)&:=(\mu_{I})_*(\se_I)^{*} \bigl( \cF\otimes_{\cO_{\tcN_\varnothing}}\cO_{\tcN_\varnothing}(-\varsigma_{I}) \bigr), \\
\Pi^{I}(\cF)&:=(\se_{I})_*(\mu_I)^{*}(\cF)\otimes_{\cO_{\tcN_\varnothing}}\cO_{\tcN_\varnothing}(\varsigma_{I}-2\rho_I)\langle - n_I \rangle,
\end{align*}
where $\rho_I$ is the halfsum of the positive roots which belong to the sublattice of $\Z\fR$ generated by the simple roots $\alpha_s$ with $s \in I$. (Here, as in~\cite{prinblock}, all the functors are understood to be \emph{derived}.)

\begin{lem}
\label{lem:Pi-exactness-exotic}
 The functor $\Pi^I$ is t-exact with respect to both the exotic and repre\-sen\-tation-theoretic t-structures, and it does not kill any nonzero object.
\end{lem}

For the t-exactness statement in this lemma, one should equip both categories $\Db \Coh^{\dot G \times \Gm}(\tcN_\varnothing)$ and $\Db \Coh^{\dot G \times \Gm}(\tcN_I)$ with the exotic t-structure, or both with the representation-theoretic t-structure. In the proof below we denote by $W_I \subset W$ the subgroup generated by $I$.

\begin{proof}
 To show that $\Pi^I$ is left t-exact, by definition it suffices to show that the object $\Pi^I(\nabla_I(\mu))$ belongs to $\Db \Coh^{\dot G \times \Gm}(\tcN_\varnothing)^{\geq 0}$ and to $\Db \Coh^{\dot G \times \Gm}(\tcN_\varnothing)^{\geq 0}_{\mathrm{RT}}$. In view of~\cite[Proposition~4(c)]{bez:ctm}, we require that 
 \begin{equation}\label{eqn:Pi-exactness-exotic}
 \Hom_{\Db \Coh^{\dot G \times \Gm}(\tcN_\varnothing)} \bigl( \Delta_\varnothing(\lambda) \la n\ra[r], \Pi^I(\nabla_I(\mu)) \bigr)=0
 \end{equation}
 for any $\lambda \in \bX$ and any $n, r \in \Z$ with $r>0$ (for the exotic case) or $r-n > 0$ (for the representation-theoretic case). By adjunction (see~\cite[Remark 9.5]{prinblock}), we have
\[
 \Hom(\Delta_\varnothing(\lambda)\la n \ra[r], \Pi^I(\nabla_I(\mu)))  \cong\Hom(\Pi_I(\Delta_\varnothing(\lambda))\langle n_{I} + n\rangle [n_{I}+r], \nabla_I(\mu)).
\]
By~\cite[Corollary~9.21 and Proposition 9.24(2)]{prinblock}, we have
\[
\Pi_I(\Delta_\varnothing(\lambda))\langle n_{I} + n\rangle [n_{I}+r] \cong
\begin{cases}
\Delta_I(\lambda')\la m+n\ra [m+r] & \text{for some $\lambda' \in \bXpp_I$ and $m \ge 0$ if $\lambda \in W_I\bXpp_I$;} \\
0 & \text{if $\lambda \notin W_I\bXpp_I$.}
\end{cases}
\]
Using~\eqref{eqn:hom-exc}, we deduce that~\eqref{eqn:Pi-exactness-exotic} holds under the present assumptions on $n$ and $r$, so $\Pi^I$ is left t-exact. An analogous argument using the adjunction $(\Pi^I\langle -n_{I}\rangle [-n_{I}], \Pi_I)$ shows that $\Pi^I$ is also right t-exact.

The fact that $\Pi^I$ does not kill any nonzero object follows from similar arguments: if $\cF$ is a nonzero object in $\Db \Coh^{\dot G \times \Gm}(\tcN_I)$, then there exists $\lambda \in \bXpp_I$ and $r,n \in \Z$ such that
\[
 \Hom_{\Db \Coh^{\dot G \times \Gm}(\tcN_I)} \bigl( \Delta_I(\lambda) \la n\ra[r], \cF \bigr) \neq 0.
\]
As above we have
\[ \Hom_{\Db \Coh^{\dot G \times \Gm}(\tcN_\varnothing)} \bigl( \Delta_\varnothing(\lambda) \la n\ra[r], \Pi^I(\cF) \bigr)
 \cong \Hom_{\Db \Coh^{\dot G \times \Gm}(\tcN_I)} \bigl( \Delta_I(\lambda) \la n\ra[r], \cF \bigr) \neq 0,
\]
so that $\Pi^I(\cF) \neq 0$.
\qed
\end{proof}

\begin{lem}
\label{lem:Pi-reptheory}
 The functor $\Pi_I$ is t-exact with respect to the repre\-sen\-tation-theoretic t-structure.
\end{lem}
\begin{proof}
We must show that $\Pi_I(\Delta_\varnothing(\lambda))$ belongs to $\Db \Coh^{\dot G \times \Gm}(\tcN_I)_{\mathrm{RT}}^{\leq 0}$ and that $\Pi_I(\nabla_\varnothing(\lambda))$ belongs to $\Db \Coh^{\dot G \times \Gm}(\tcN_I)_{\mathrm{RT}}^{\geq 0}$.  Both of these assertions follow from~\cite[Proposition~9.24]{prinblock} and the definition of the representation-theoretic t-structure.
\qed
\end{proof}

\subsection{Mixed derived category of the affine Grassmannian}
\label{ss:mixed-Gr}

The proof of Theorem~\ref{thm:main} will require a ``translation'' of the problem to a setting involving constructible sheaves on affine flag varieties. This will require in particular an equivalence of categories obtained in~\cite{ar:agsr, mr:etsps} (which adapts a result obtained by Arkhipov--Bezrukavnikov--Ginzburg~\cite{abg:qglg} for characteristic-$0$ coefficients), that we explain now.

Let $\mathscr{K} = \C( \hspace{-0.5pt} (\mathrm{t}) \hspace{-0.5pt} )$ be the field of formal Laurent series in an indeterminate $\mathrm{t}$, and let $\mathscr{O} = \C[ \hspace{-0.5pt} [\mathrm{t}] \hspace{-0.5pt} ]$ be the ring of formal power series in $\mathrm{t}$.  Let $\dot G^\vee$ be the complex reductive algebraic group which is Langlands dual to $\dot G$. By definition this group comes with a maximal torus $\dot T^\vee$ whose cocharacter lattice is $\bX$, and such that the root system of $(\dot G^\vee, \dot T^\vee$) identifies with the coroot system of $(\dot G, \dot T)$. We will also consider the Borel subgroup $\dot B^\vee \subset \dot G^\vee$ containing $\dot T^\vee$ and associated with the negative coroots of $\dot G$.

Let $\Waff:=W\ltimes\bX$ be the extended affine Weyl group and let $\WaffCox\subset\Waff$ be the ``true'' affine Weyl group, i.e.~the subgroup $W\ltimes\Z\fR$. (The image of $\lambda \in \bX$ in $\Waff$ will be denoted $t_\lambda$.) Then $\WaffCox$ has a natural structure of Coxeter group (such that $S$ is a subset of the set of simple reflections), and whose length function extends to $\Waff$. We will denote by $\Waffmin \subset \Waff$ the subset of elements $w$ which are minimal in the coset $W w$. This subset is in bijection with $\bX$, via the map sending $\lambda \in \bX$ to the minimal element $w_\lambda$ in $W t_\lambda$.
(See e.g.~\cite[\S 2.2]{mr:etspc} for details and references on this subject.)

Consider the (left version of the) affine Grassmannian
\[
 \Gr' := \dot G^\vee(\mathscr{O}) \backslash \dot G^\vee(\mathscr{K}).
\]
  We denote by $\Iw \subset \dot G^\vee(\mathscr{O})$ the Iwahori subgroup associated with $\dot B^\vee$, i.e.~the inverse image of $\dot B^\vee$ under the ``evaluation at $\mathrm{t}=0$'' morphism $\dot G^\vee(\mathscr{O}) \to \dot G^\vee$. We consider the action of $\Iw$ on $\Gr'$ induced by right multiplication in $\dot G^\vee(\mathscr{K})$. The orbits for this action are parametrized in a natural way by the subset $\Waffmin \subset \Waff$, and we will denote by $\Gr'_w$ the orbit associated with $w$.
  Since each of these orbits is isomorphic to an affine space, following~\cite{modrap2} we can consider the \emph{mixed derived category}
  \[
   \Dmix_{(\Iw)}(\Gr',\bk):= \Kb\Par_{(\Iw)}(\Gr',\bk)
  \]
and its perverse t-structure, where $\Par_{(\Iw)}(\Gr',\bk)$ is the category of parity complexes on $\Gr'$ with respect to the stratification by $\Iw$-orbits, in the sense of~\cite{jmw}. In particular, this theory provides standard and costandard (mixed) perverse sheaves\footnote{In the general setting for the construction from~\cite{modrap2} it is not known if the standard and costandard objects are perverse; but this is true in the case of affine Grassmannians by~\cite[Corollary~A.8]{modrap2}.} $\Delta^{\mix}_w$ and $\nabla^{\mix}_w$ for all $w \in \Waffmin$, and indecomposable tilting perverse sheaves $\tilt^{\mix}_w$. We will denote by $\{1\}$ the autoequivalence of $\Dmix_{(\Iw)}(\Gr',\bk)$ induced by the cohomological shift in $\Par_{(\Iw)}(\Gr',\bk)$, and by $[1]$ the usual shift of complexes; then the ``Tate twist'' $\langle 1 \rangle := \{-1\}[1]$ is t-exact for the perverse t-structure.

The results of~\cite{ar:agsr, mr:etsps} provide an equivalence of triangulated categories
\[
 \Psi : \Db \Coh^{\dot G \times \Gm}(\tcN_\varnothing) \xrightarrow{\sim} \Dmix_{(\Iw)}(\Gr', \bk)
\]
which satisfies $\Psi \circ \langle 1 \rangle \cong \langle 1 \rangle[-1] \circ \Psi$ and
\begin{equation}
\label{eqn:Psi-DN}
\Psi(\Delta_\varnothing(\lambda)) \cong \Delta^{\mix}_{w_\lambda}, \quad \Psi(\nabla_\varnothing(\lambda)) \cong \nabla^{\mix}_{w_\lambda}.
\end{equation}
From these properties we see that $\Psi$ is t-exact if $\Db \Coh^{\dot G \times \Gm}(\tcN_\varnothing)$ is equipped with the representation-theoretic t-structure and $\Dmix_{(\Iw)}(\Gr', \bk)$ is equipped with the perverse t-structure. 

\begin{rmk}{\rm
 See~\cite[Remark~11.3]{prinblock} for a comparison of the conventions used in~\cite{ar:agsr} and in~\cite{mr:etsps}.  The assumptions in~\S\ref{ss:notation} come from~\cite{ar:agsr}; in~\cite{mr:etsps}, there are slightly more restrictive conditions on the group.  Note that~\cite{ar:agsr} and parts of~\cite{mr:etsps} work instead with the ``left coset'' affine Grassmannian $\dot G^\vee(\mathscr{K}) / \dot G^\vee(\mathscr{O})$, but as explained in~\cite[\S 3.2]{mr:etsps} it is straightforward to pass back and forth between this variety and $\Gr'$.  We will use $\Gr'$ for now because it is more convenient for the arguments in Section~\ref{s:mainproof}, but in Section~\ref{s:singular-fm}, we will switch to (a positive characteristic analogue of) $\dot G^\vee(\mathscr{K}) / \dot G^\vee(\mathscr{O})$.}
\end{rmk}

The triangulated category $\Dmix_{(\Iw)}(\Gr', \bk)$ also admits a second interesting t-struc\-ture, called the \emph{adverse} t-structure, defined in~\cite[\S A.2]{ar:agsr}. This t-structure consists of the subcategories
\begin{align*}
 {}^a \hspace{-1pt} \Dmix_{(\Iw)}(\Gr', \bk)^{\leq 0} &= \langle \Delta^{\mix}_w \langle n \rangle [m] : w \in \Waffmin, \, n,m \in \Z \text{ with } n+m \geq 0 \rangle_{\mathrm{ext}}; \\
 {}^a \hspace{-1pt} \Dmix_{(\Iw)}(\Gr', \bk)^{\geq 0} &= \langle \nabla^{\mix}_w \langle n \rangle [m] : w \in \Waffmin, \, n,m \in \Z \text{ with } n+m \leq 0 \rangle_{\mathrm{ext}}.
\end{align*}
From this definition we see that the functor $\langle -1 \rangle[1]$ is t-exact with respect to the adverse t-structure and
(using also~\eqref{eqn:Psi-DN}) that $\Psi$ is t-exact if $\Db \Coh^{\dot G \times \Gm}(\tcN_\varnothing)$ is equipped with the exotic t-structure and $\Dmix_{(\Iw)}(\Gr', \bk)$ is equipped with the adverse t-struc\-ture.

\subsection{Highest weight structure}

The following analogue of Theorem~\ref{thm:main} for the representation-theoretic t-structure turns out to be much easier to prove.

\begin{lem}
\label{lem:rt-delta-heart}
 The objects $\Delta_I(\lambda)$ and $\nabla_I(\lambda)$ ($\lambda \in \bXpp_I$) belong to $\grRep(\tcN_I)$.
\end{lem}

\begin{proof}
Let us first treat the case where $I = \varnothing$.  In this case, we can use the equivalence of categories $\Psi$ introduced in~\S\ref{ss:mixed-Gr}. Since this equivalence 
takes the representation-theoretic t-structure on $\Db\Coh^{\dot G \times \Gm}(\tcN_\varnothing)$ to the adverse t-structure on $\Dmix_{(\Iw)}(\Gr', \bk)$, it suffices to prove that the objects $\Delta^{\mix}_w$ and $\nabla^{\mix}_w$ ($w \in \Waffmin$) belong to the heart of the adverse t-structure.  This claim holds by~\cite[Propo\-si\-tion~A.16]{ar:agsr}.

Now suppose that $I \ne \varnothing$, and let $\lambda \in \bXpp_I$.  By~\cite[Proposition~9.24]{prinblock} we have $\Delta_I(\lambda) \cong \Pi_I(\Delta_\varnothing(\lambda))\la n_I\ra[n_I]$. By Lemma~\ref{lem:Pi-reptheory} and the previous paragraph, we conclude that $\Delta_I(\lambda) \in \grRep(\tcN_I)$.  Similar reasoning applies to $\nabla_I(\lambda)$.
\qed
\end{proof}

\begin{rmk}{\rm
 Under the assumption that $\ell$ is bigger than the Coxeter number of $\dot G$, one can alternatively prove Lemma~\ref{lem:rt-delta-heart} by using the fact that the functor $\Phi_I$ considered in~\S\ref{ss:relation} is t-exact if $\Db \Coh^{\dot G \times \Gm}(\tcN_I)$ is endowed with the representation-theoretic t-structure and $\Db \Rep_I(G)$ with its tautological t-structure, and sends standard, resp.~costandard, objects to standard, resp.~costandard, objects.}
\end{rmk}

As in~\S\ref{ss:main}, this lemma implies that the category $\grRep(\tcN_I)$ is a graded highest weight category in the sense of~\cite[\S 7]{riche-hab}, with weight poset $(\bXpp_I, \leq')$, standard objects $(\Delta_I(\lambda): \lambda \in \bXpp_I)$, costandard objects $(\nabla_I(\lambda): \lambda \in \bXpp_I)$, and shift functor $\langle 1 \rangle[1]$. In particular, we can consider the tilting objects in $\grRep(\tcN_I)$. The indecomposable tilting object associated with $\lambda$ will be denoted $\TilRT_I(\lambda)$, and the multiplicity of a standard, resp.~costandard, object $\Delta_I(\lambda) \langle m \rangle [m]$, resp.~$\nabla_I(\lambda) \langle m \rangle [m]$, in a tilting object $\cT$ will be denoted
\[
 (\cT : \Delta_I(\lambda) \langle m \rangle [m] ), \quad \text{resp.} \quad (\cT : \nabla_I(\lambda) \langle m \rangle [m] ).
\]

Note that, in the case $I=\varnothing$, the isomorphisms in~\eqref{eqn:Psi-DN} imply
that for $\lambda \in \bX$ we have
\begin{equation}
 \label{eqn:Psi-tilting}
 \Psi(\TilRT_\varnothing(\lambda)) \cong \tilt^{\mix}_{w_\lambda}.
\end{equation}

\begin{rmk}{\rm
 In the case $I=\varnothing$, the object
 $\TilRT_\varnothing(\lambda)$ coincides with the indecomposable \emph{exotic parity object} in $\Db \Coh^{\dot G \times \Gm}(\tcN_\varnothing)$ associated with $\lambda$ studied in~\cite{ahr}. Once Corollary~\ref{cor:hw} is established, this claim can be generalized to any subset $I \subset S$.}
\end{rmk}

\subsection{Compatibility with translation functors}
\label{ss:compatibility}

We are now in a position to refine some claims from Lemma~\ref{lem:Pi-exactness-exotic}.

\begin{lem}
\label{lem:translation-Delta}
For any $\lambda \in \bXpp_I$, the object $\Pi^I(\Delta_I(\lambda))$, resp.~$\Pi^I(\nabla_I(\lambda))$, belongs to $\grRep(\tcN_\varnothing)$,
and it admits a filtration whose subquotients are the objects of the form $\Delta_\varnothing(\mu) \langle m \rangle [m]$, resp.~$\nabla_\varnothing(\mu) \langle -m \rangle [-m]$, where $\mu \in W_I ( \lambda)$ and $m$ is the length of the unique element $w \in W_I$ such that $\mu=w(\lambda)$ (each appearing once).
\end{lem}

\begin{proof}
The first assertion follows from Lemma~\ref{lem:Pi-exactness-exotic} and Lemma~\ref{lem:rt-delta-heart}. We will prove the second one for the object $\Delta_I(\lambda)$; the case of $\nabla_I(\lambda)$ is similar.

By general properties of graded highest weight categories (see e.g.~\cite[\S 7.4]{riche-hab}), we know that an object $\cF \in \grRep(\tcN_\varnothing)$ admits a filtration whose subquotients are standard objects if and only if
\begin{equation}
\label{eqn:transl1}
\Ext^1_{\grRep(\tcN_\varnothing)} ( \cF, \nabla_\varnothing(\mu)\la n\ra[n] )
= \Hom_{\Db \Coh^{\dot G \times \Gm}(\tcN_\varnothing)} ( \cF, \nabla_\varnothing(\mu)\la n\ra[n+1] ) = 0
\end{equation}
for all $\mu \in \bX$ and all $n \in \Z$.  Moreover, if $\cF$ admits such a filtration, then the number of occurences of a specific standard object $\Delta_\varnothing(\mu)\la n\ra [n]$ as a subquotient in such a filtration is the dimension of
\begin{equation}
\label{eqn:transl2}
\Hom_{\Db \Coh^{\dot G \times \Gm}(\tcN_\varnothing)}(\cF, \nabla_{\varnothing}(\mu)\la n\ra[n]).
\end{equation}

We apply this criterion to $\cF=\Pi^I(\Delta_I(\lambda))$.
As in the proof of Lemma~\ref{lem:Pi-exactness-exotic}
we have
\begin{multline*}
\Hom(\Pi^I(\Delta_I(\lambda)),\nabla_\varnothing(\mu)\langle m \rangle [n])  \cong \Hom(\Delta_I(\lambda),\Pi_I(\nabla_\varnothing(\mu))\langle m-n_I \rangle [n-n_I] ) \\
\cong
\begin{cases}
\Hom(\Delta_I(\lambda), \nabla_I(w\mu)\langle m - \ell(w) \rangle [n - \ell(w)]) & \text{if $w \in W_I$ and $w\mu \in \bXpp_I$;} \\
0 & \text{if $w\mu \notin \bXpp_I$ for all $w \in W_I$.}
\end{cases}
\end{multline*}
Using~\eqref{eqn:hom-exc}, we deduce that $\Hom(\Pi^I(\Delta_I(\lambda)),\nabla_\varnothing(\mu)\langle m \rangle [n])$ vanishes unless $m = n = \ell(w)$ and $\lambda = w \mu$ for some $w \in W_I$, and is $1$-dimensional otherwise.  In particular, we have confirmed~\eqref{eqn:transl1} for $\cF = \Pi^I(\Delta_I(\lambda))$, and we have shown that the space~\eqref{eqn:transl2} is $1$-dimensional if $\mu \in W_I ( \lambda)$ and $m$ is the length of the unique element $w \in W_I$ such that $\mu=w(\lambda)$, and $0$-dimensional otherwise.
\qed
\end{proof}

\begin{lem}
\label{lem:Pi-TilRT}
 For any $\lambda \in \bXpp_I$, the object $\Pi^I(\TilRT_I(\lambda))$ is a tilting object in $\grRep(\tcN_\varnothing)$. It admits $\TilRT_\varnothing(\lambda)$ as a direct summand, and all its other direct summands are of the form $\TilRT_\varnothing(\nu) \langle m \rangle [m]$ with $\nu <' \lambda$ and $\nu \notin W_I(\lambda)$.
\end{lem}

\begin{proof}
 It follows from Lemma~\ref{lem:translation-Delta} that the functor $\Pi^I$ sends tilting objects in $\grRep(\tcN_I)$ to tilting objects in $\grRep(\tcN_\varnothing)$. Hence indeed $\Pi^I(\TilRT_I(\lambda))$ is tilting, and Lemma~\ref{lem:translation-Delta} gives us information on the multiplicities $(\Pi^I(\TilRT_I(\lambda)) : \Delta_\varnothing(\nu) \langle m \rangle [m])$. More precisely, using also the construction of the order $\leq'$ (see in particular~\cite[Equation~(9.9)]{prinblock}), this information shows that
 \[
 (\Pi^I(\TilRT_I(\lambda)) : \Delta_\varnothing(\nu) \langle m \rangle [m]) \neq 0 \quad \Rightarrow \quad \nu \leq' \lambda,
 \]
 and that $(\Pi^I(\TilRT_I(\lambda)) : \Delta_\varnothing(\lambda) \langle m \rangle [m])=\delta_{m,0}$. We deduce that $\TilRT_\varnothing(\lambda)$ is a direct summand of $\Pi^I(\TilRT_I(\lambda))$, with multiplicity $1$, and that all the other indecomposable direct summands are of the form $\TilRT_\varnothing(\mu) \langle m \rangle [m]$ with $\mu <' \lambda$. If an object $\TilRT_\varnothing(\mu) \langle m \rangle [m]$ with $\mu \in W_I(\lambda) \smallsetminus \{\lambda\}$ was also a direct summand, then we would have
 \[
  (\Pi^I(\TilRT_I(\lambda)) : \Delta_\varnothing(\mu) \langle m \rangle [m]) \neq 0 \quad \text{and} \quad (\Pi^I(\TilRT_I(\lambda)) : \nabla_\varnothing(\mu) \langle m \rangle [m]) \neq 0.
 \]
The first condition would imply that $m>0$, and the second one that $m<0$; a contradiction.
\qed
\end{proof}

\begin{rmk}{\rm
We expect (but do not prove in general) that in fact $\Pi^I(\TilRT_I(\lambda)) \cong \TilRT_\varnothing(\lambda)$. See
 Remark~\ref{rmk:Pi^I-tilting} for more details.}
\end{rmk}

\section{Proof of Theorem~\ref{thm:main}}
\label{s:mainproof}

\subsection{Reduction to a claim about adverse sheaves}
\label{ss:adverse}

For $\lambda \in \bXpp_I$, set
\[
 X_\lambda := \bigsqcup_{\mu \leq' \lambda} \Gr'_{w_\mu}, \quad U_\lambda := \bigsqcup_{\mu \in W_I (\lambda)} \Gr'_{w_\mu}.
\]
By definition of the order $\leq'$, $X_\lambda$ is a closed subvariety of $\Gr'$, and $U_\lambda$ is open in $X_\lambda$. We will consider the open and closed embeddings
\[
 j_\lambda : U_\lambda \hookrightarrow X_\lambda, \qquad i_\lambda : X_\lambda \smallsetminus U_\lambda \hookrightarrow X_\lambda.
\]

The definition of the mixed derived category in~\cite{modrap2} applies to locally closed unions of $\Iw$-orbits in $\Gr'$ also; in particular we can consider the categories $\Dmix_{(\Iw)}(X_\lambda,\bk)$ and $\Dmix_{(\Iw)}(U_\lambda,\bk)$, these categories possess perverse t-structures, and they are related by functors $(i_\lambda)_*$, $(i_\lambda)^*$, $(i_\lambda)^!$, $(j_\lambda)_*$, $(j_\lambda)_!$, $(j_\lambda)^*$ which satisfy the usual adjunction properties. Moreover, we have a fully faithful and t-exact pushforward functor $\Dmix_{(\Iw)}(X_\lambda,\bk) \to \Dmix_{(\Iw)}(\Gr',\bk)$ associated with the closed embedding $X_\lambda \to \Gr'$, whose essential image contains $\tilt^{\mix}_{w_\lambda}$. Therefore we may (and will) consider this object as belonging to $\Dmix_{(\Iw)}(X_\lambda,\bk)$.

In~\S\ref{ss:proof-adv} below we will prove the following claim.

\begin{lem}
\label{lem:adverse}
 For any $\lambda \in \bXpp_I$, the objects
  $(j_\lambda)_! (j_\lambda)^* \tilt^{\mix}_{w_\lambda}$ and $(j_\lambda)_* (j_\lambda)^* \tilt^{\mix}_{w_\lambda}$
belong to the heart of the adverse t-structure.
\end{lem}

In the rest of this subsection we show that Lemma~\ref{lem:adverse} implies Theorem~\ref{thm:main}. We fix $\lambda \in \bXpp_I$, and will prove that $\Delta_I(\lambda)$ belongs to $\ExCoh(\tcN_I)$. The case of $\nabla_I(\lambda)$ can be treated similarly.

By general properties of highest weight categories (see e.g.~\cite[Theorem~7.14]{riche-hab}), we have an exact sequence
\[
 \Delta_I(\lambda) \hookrightarrow \TilRT_I(\lambda) \twoheadrightarrow \mathrm{coker}
\]
in $\grRep(\tcN_I)$, where $\mathrm{coker}$ is an extension of objects of the form $\Delta_I(\mu)\langle m \rangle [m]$ with $m \in \Z$ and $\mu \in \bXpp_I$ such that $\mu <' \lambda$. 
Applying the exact functor $\Pi^I$ (see Lemma~\ref{lem:Pi-exactness-exotic}) we deduce an exact sequence 
\begin{equation}
\label{eqn:triangle-Pi^I}
 \Pi^I(\Delta_I(\lambda)) \hookrightarrow \Pi^I(\TilRT_I(\lambda)) \twoheadrightarrow \Pi^I(\mathrm{coker})
\end{equation}
in $\grRep(\tcN_\varnothing)$.
In view of Lemma~\ref{lem:Pi-TilRT}, we can choose an isomorphism
\[
 \Pi^I(\TilRT_I(\lambda)) \cong \TilRT_\varnothing(\lambda) \oplus \cT,
\]
where $\cT$ is a direct sum of objects of the form $\TilRT_\varnothing(\nu) \langle m \rangle [m]$ with $\nu <' \lambda$ and $\nu \notin W_I(\lambda)$. Using Lemma~\ref{lem:translation-Delta} we see that $\Hom(\Pi^I(\Delta_I(\lambda)),\cT)=0$, so that the first arrow in~\eqref{eqn:triangle-Pi^I} factors through an embedding $\Pi^I(\Delta_I(\lambda)) \hookrightarrow \TilRT_\varnothing(\lambda)$ whose cokernel is a direct summand of the third term in~\eqref{eqn:triangle-Pi^I}. In conclusion, we have constructed a distinguished triangle
\begin{equation}
\label{eqn:triangle-Pi^I-2}
 \Pi^I(\Delta_I(\lambda)) \to \TilRT_\varnothing(\lambda) \to \cF \xrightarrow{[1]}
\end{equation}
in $\Db \Coh^{\dot G \times \Gm}(\tcN_\varnothing)$, whose first term belongs to the triangulated subcategory 
generated by the objects of the form $\Delta_\varnothing(\mu) \langle n \rangle$ with $\mu \in W_I ( \lambda)$, and whose third term belongs to the triangulated subcategory generated by the objects of the form $\Delta_\varnothing(\nu) \langle n \rangle$ with $\nu <' \lambda$ and $\nu \notin W_I(\lambda)$.

Applying $\Psi$ to~\eqref{eqn:triangle-Pi^I-2} and using~\eqref{eqn:Psi-tilting} we obtain a distinguished triangle
\begin{equation}
\label{eqn:triangle-proof-1}
 \Psi \bigl( \Pi^I(\Delta_I(\lambda)) \bigr) \to \tilt^{\mix}_{w_\lambda} \to \Psi ( \cF ) \xrightarrow{[1]}.
\end{equation}
The comments above and~\eqref{eqn:Psi-DN} show that all three objects in this triangle are supported on $X_\lambda$, and that
\begin{equation}
 \label{eqn:properties-triangle-1}
 (i_\lambda)^* \Psi \bigl( \Pi^I(\Delta_I(\lambda)) \bigr)=0, \quad (j_\lambda)^* \Psi ( \cF)=0.
\end{equation}

Now, consider the natural distinguished triangle
\begin{equation}
\label{eqn:triangle-proof-2}
 (j_\lambda)_! (j_\lambda)^* \tilt^{\mix}_{w_\lambda} \to \tilt^{\mix}_{w_\lambda} \to (i_\lambda)_* (i_\lambda)^* \tilt^{\mix}_{w_\lambda} \xrightarrow{[1]},
\end{equation}
where the first two arrows are the adjunction maps.
Adjunction and the first equality in~\eqref{eqn:properties-triangle-1} show that the composition of the first arrow in~\eqref{eqn:triangle-proof-1} with the second arrow in~\eqref{eqn:triangle-proof-2} vanishes. In view of~\cite[Proposition~1.1.9]{bbd}, this implies that there exists a unique morphism of triangles
\[
\begin{tikzcd}
\Psi \bigl( \Pi^I(\Delta_I(\lambda)) \bigr) \ar[r] \ar[d] & \tilt^{\mix}_{w_\lambda} \ar[r] \ar[d] & \Psi \bigl( \Pi^I(\mathrm{coker}) \bigr) \ar[r, "{[1]}"] \ar[d] & \ \\
 (j_\lambda)_! (j_\lambda)^* \tilt^{\mix}_{w_\lambda} \ar[r] & \tilt^{\mix}_{w_\lambda} \ar[r] & (i_\lambda)_* (i_\lambda)^* \tilt^{\mix}_{w_\lambda} \ar[r, "{[1]}"] & \ 
\end{tikzcd}
\]
whose middle morphism is the identity. (Here, the upper line is~\eqref{eqn:triangle-proof-1}, and the lower line is~\eqref{eqn:triangle-proof-2}.) Similar considerations using the second property in~\eqref{eqn:properties-triangle-1} produce a morphism of triangles in the reverse direction, and applying the uniqueness claim in~\cite[Proposition~1.1.9]{bbd} to both compositions of these morphisms we see that they are isomorphisms, inverse to each other.

In particular, we have shown that there exists an isomorphism
\[
 \Psi \bigl( \Pi^I(\Delta_I(\lambda)) \bigr) \cong (j_\lambda)_! (j_\lambda)^* \tilt^{\mix}_{w_\lambda}.
\]
By Lemma~\ref{lem:adverse} the right-hand side belongs to the heart of the adverse t-structure. By t-exactness of $\Psi$ (see~\S\ref{ss:mixed-Gr}), it follows that $\Pi^I(\Delta_I(\lambda))$ belongs to the heart of the exotic t-structure. In view of Lemma~\ref{lem:Pi-exactness-exotic}, this shows that $\Delta_I(\lambda)$ belongs to the heart of the exotic t-structure, and finishes the proof of Theorem~\ref{thm:main}.

\subsection{Iwahori--Whittaker mixed derived categories}
\label{ss:IW-Dmix}

The proof of Lemma~\ref{lem:adverse} will require a new ``translation'', via the Koszul duality from~\cite{amrw} (which adapts a characteristic-$0$ construction due to Bezrukavnikov--Yun~\cite{by}), to a setting involving Whittaker parity complexes on affine flag varieties. This will require in particular a passage from the classical topology to the \'etale topology. Since \'etale sheaves are defined only for certain choices of coefficients, in the next two subsections we replace $\bk$ by a large enough finite field of characteristic $\ell$ (or by a finite extension of $\mathbb{Q}_{\ell'}$ for some prime $\ell'$ if $\ell=0$), that we will still denote by $\bk$ for simplicity. Lemma~\ref{lem:adverse} of course makes sense for such coefficients, and this variant will imply the lemma as stated in~\S\ref{ss:adverse}.

The general theory of Whittaker parity complexes on partial flag varieties of Kac--Moody groups is developed in Appendix~\ref{sec:appendix} (after partial results obtained in~\cite{rw, amrw}). In this subsection we explain how to apply these general results in the present case of affine flag varieties.

Let $\F$ be an algebrai\-cally closed field such that $\mathrm{char}(\F)=p>0$ and $p\neq \ell$ (or $p \neq \ell'$ if $\ell=0$). We now consider the connected reductive $\F$-group $\dot G^{\vee}_{\F}$ which is Langlands dual to $\dot G$, and its maximal torus $\dot T^\vee_\F$ and Borel subgroup $\dot B^\vee_\F$, defined as in~\S\ref{ss:mixed-Gr}. We will also denote by $H$ the simply-connected cover of the derived subgroup of $\dot G^{\vee}_{\F}$. We can then consider the groups $H(\F( \hspace{-0.5pt} (\mathrm{t}) \hspace{-0.5pt} ))$ and $H(\F[ \hspace{-0.5pt} [\mathrm{t}] \hspace{-0.5pt} ])$, the Iwahori subgroup $\Iw_H \subset H(\F[ \hspace{-0.5pt} [\mathrm{t}] \hspace{-0.5pt} ])$, and the affine flag variety
\[
 \Fl:=H(\F( \hspace{-0.5pt} (\mathrm{t}) \hspace{-0.5pt} ))\slash\Iw_{H}.
\]
(By our conventions, this ind-variety is connected.)

As in~\cite[\S 11.7]{rw} and~\cite[\S 7.2]{amrw} we can consider the ``Iwahori--Whittaker'' derived category $\Db_{\mathcal{IW}}(\Fl, \bk)$ and its ``mixed'' variant $\Dmix_{\mathcal{IW}}(\Fl, \bk)$ (defined again as an appropriate bounded homotopy category of parity complexes). These categories are defined using the action of an Iwahori subgroup $\Iw^+_H$ associated with a \emph{positive} Borel subgroup in $H$. The orbits of $\Iw^+_H$ on $\Fl$ are parametrized in a natural way by $\WaffCox$ and we denote the orbit associated with $w$ by $\Fl_w$. If we set $\WaffminCox:= \Waffmin \cap \WaffCox$, then the indecomposable parity complexes in $\Db_{\mathcal{IW}}(\Fl, \bk)$ are naturally parametrized by $\WaffminCox \times \Z$, and we will denote by $\cE_w^{\IW}$ the object corresponding to $(w,0)$. The complex concentrated in degree $0$, and with degree-$0$ term $\cE^{\IW}_w$, is an object of $\Dmix_{\mathcal{IW}}(\Fl, \bk)$, which will be denoted $\cE^{\IW,\mix}_w$.

The cohomological shift in the triangulated category $\Db_{\mathcal{IW}}(\Fl, \bk)$ restricts to an autoequivalence of the subcategory of parity complexes, and thus induces an autoequivalence of $\Dmix_{\mathcal{IW}}(\Fl, \bk)$. This autoequivalence will be denoted by $\{1\}$.

To $I\subset S$ we can also associate a parabolic affine flag variety
\[
 \Fl_I:=H(\F( \hspace{-1pt} (\mathrm{t}) \hspace{-1pt} ))\slash\cP^{I}_{H},
\]
where $\cP^{I}_{H} \subset H(\F[ \hspace{-1pt} [\mathrm{t}] \hspace{-1pt} ])$ is the parahoric subgroup corresponding to the parabolic subgroup of $H$ containing the negative Borel subgroup and associated which $I$.
The natural projection $q_I : \Fl \to \Fl_I$ is a smooth, proper morphism. In the same way as for $\Fl$, we can consider the Iwahori--Whittaker derived category $\Db_{\mathcal{IW}}(\Fl_I, \bk)$ and its ``mixed'' variant $\Dmix_{\mathcal{IW}}(\Fl_I, \bk)$. The latter category admits a natural perverse t-structure, defined by the same procedure as in~\cite{modrap2}; see Appendix~\ref{sec:appendix} for details.

\begin{rmk}{\rm
 In Appendix~\ref{sec:appendix}, for completeness we work in the setting of (partial) flag varieties of Kac--Moody groups. It seems very likely that the partial affine flag varieties considered above are special cases of (products of) flag varieties of (untwisted affine) Kac--Moody groups, but we do not know of a reference for this claim. (See, however,~\cite[Chap.~XIII]{kumar} for similar claims when the base field is $\mathbb{C}$, and~\cite[\S 9.f]{pr} for the case $I=\varnothing$.) In any case, all the properties of flag varieties of Kac--Moody groups that we use in Appendix~\ref{sec:appendix} have well-known analogues for affine flag varieties, so that all of our results apply in this setting also.}
\end{rmk}

The orbits of $\Iw_{H}^+$ on $\Fl_I$ are parametrized in a natural way by $\WaffCox / W_I$, or equivalently by the subset of $\WaffCox$ consisting of elements $w$ which are maximal in $w W_I$.  Let $\WaffminI\subset\Waff$ be the subset 
\[
\WaffminI:=\{w\in\Waff \mid w \text{ is maximal in } wW_{I} \text{ and } wv\in\Waffmin \text{ for all } v\in W_{I}\}.
\]
Then it is well known that the orbits that support a nonzero Iwahori--Whittaker local system correspond to the subset $\WaffCox \cap \WaffminI \subset \WaffCox$.
Therefore, the indecomposable parity complexes in $\Db_{\mathcal{IW}}(\Fl_I, \bk)$ are parametrized by $(\WaffCox \cap \WaffminI) \times \Z$ (see~\S\hyperref[ss:Whittaker-appendix]{A.B} for more details). For any $w \in \WaffCox \cap \WaffminI$, we will denote by $\cE_w^{\IW,I}$ the indecomposable parity complex associated with $(w,0)$, and by $\cE_w^{\IW,I,\mix}$ the corresponding object of $\Dmix_{\mathcal{IW}}(\Fl_I, \bk)$. Similarly, for such $w$'s we have standard and costandard objects in $\Dmix_{\mathcal{IW}}(\Fl_I, \bk)$, which will be denoted $\Delta^{\IW,I,\mix}_w$ and $\nabla^{\IW,I,\mix}_w$ respectively.

\begin{rmk}
\label{rmk:max-min}{\rm
 In Appendix~\ref{sec:appendix} we label orbits in $\Fl_I$ by \emph{minimal} representatives in cosets $w W_I$ instead of maximal representatives. Thus, if $w_0^I$ is the longest element in $W_I$, then the orbit labelled by $w$ in the present section would be labelled by $w w_0^I$ with the conventions of Appendix~\ref{sec:appendix}.}
\end{rmk}

The following claim is a counterpart in the present setting of Lemma~\ref{lem:pullback-perverse}, and follows from similar considerations.

\begin{lem}
\label{lem:Whit-perverse}
 For any $w \in \WaffCox \cap \WaffminI$, the objects $(q_I)^* \Delta^{\IW,I,\mix}_w \{\ell(w_0^I)\}$ and $(q_I)^* \nabla^{\IW,I,\mix}_w \{\ell(w_0^I)\}$ belong to the heart of the perverse t-structure.
\end{lem}

\subsection{Proof of Lemma~\ref{lem:adverse}}
\label{ss:proof-adv}

Using Lemma~\ref{lem:Whit-perverse}, we can now prove Lemma~\ref{lem:adverse}. In fact we will treat the case of $(j_\lambda)_! (j_\lambda)^* \tilt^{\mix}_{w_\lambda}$; the case of $(j_\lambda)_* (j_\lambda)^* \tilt^{\mix}_{w_\lambda}$ is similar and left to the reader.

We first remark that all the connected components of $\Gr'$ are isomorphic as ind-varieties stratified by the $\Iw$-orbits. Hence we can assume that $\lambda$ belongs to $\Z\fR$, or equivalently that $w_\lambda \in \WaffCox$, or in other words that $\tilt^{\mix}_{w_\lambda}$ is supported on the connected component $(\Gr')^0$ of the base point.

Now, recall the ``Koszul duality'' equivalence
\[
 \kappa : \Dmix_{(\Iw)} \bigl( (\Gr')^0,\bk \bigr) \xrightarrow{\sim} \Dmix_{\mathcal{IW}}(\Fl, \bk)
\]
from~\cite[Theorem~7.4]{amrw}. This equivalence satisfies $\kappa \circ \langle 1 \rangle = \langle 1 \rangle [1] \circ \kappa$, and sends standard, resp.~costandard, perverse sheaves to standard, resp.~costandard, perverse sheaves (in a way compatible with labellings). Hence it is t-exact if $\Dmix_{(\Iw)}((\Gr')^0,\bk)$ is endowed with the adverse t-structure and $\Dmix_{\mathcal{IW}}(\Fl, \bk)$ is endowed with the perverse t-structure. 
By~\cite[Theorem~7.4]{amrw}, it also satisfies
\begin{equation}
\label{eqn:kappa-tilt-E}
 \kappa(\tilt^{\mix}_{w}) \cong \cE_{w}^{\IW,\mix}
\end{equation}
for any $w \in \WaffCox \cap \Waffmin$.

Fix $\lambda \in \bXpp_I$, and recall the distinguished triangle~\eqref{eqn:triangle-proof-2} in $\Dmix_{(\Iw)} \bigl( (\Gr')^0,\bk \bigr)$.
Applying $\kappa$ and using~\eqref{eqn:kappa-tilt-E} we deduce a distinguished triangle 
\begin{equation}
\label{eqn:triangle-kappa}
\kappa((j_\lambda)_!(j_\lambda)^*\tilt^{\mix}_{w_\lambda})\to \cE_{w_\lambda}^{\IW,\mix} \to \kappa((i_\lambda)_*(i_\lambda)^*\tilt^{\mix}_{w_\lambda})\xrightarrow{[1]}.
\end{equation}

On the partial affine flag variety $\Fl_I$, let $(\Fl_I)_{w_\lambda}$ be the $\Iw_{H}^+$-orbit corresponding to $w_\lambda$, and consider the embeddings
\[
j^I_\lambda:(\Fl_I)_{w_\lambda}\hookrightarrow\overline{(\Fl_I)_{w_\lambda}}, \qquad i^I_\lambda:\overline{(\Fl_I)_{w_\lambda}}\smallsetminus(\Fl_I)_{w_\lambda}\hookrightarrow\overline{(\Fl_I)_{w_\lambda}}.
\]
As in~\S\ref{ss:adverse} we have pullback and pushforward functors between mixed derived categories associated with these maps, and
using these functors we obtain a canonical distinguished triangle
\begin{equation}
\label{eqn:triangle-Fl_I}
\Delta^{\IW,I,\mix}_{w_\lambda}\to\cE^{\IW,I,\mix}_{w_\lambda}\to(i_\lambda^{I})_*(i_\lambda^{I})^*\cE^{\IW,I,\mix}_{w_\lambda}\xrightarrow{[1]}.
\end{equation}
Now by~\cite[Lemma~10.2]{prinblock} we have $w_\lambda \in \WaffminI$. Hence by Lemma~\ref{lem:pullback-parity} (see also Remark~\ref{rmk:max-min}) we have an isomorphism $\cE_{w_\lambda}^{\IW,\mix}\cong(q_I)^* \cE_{w_{\lambda}}^{\IW,I,\mix} \{\ell(w_0^I)\}$;
therefore applying the triangulated functor $(q_I)^*\{\ell(w_0^I)\}$ to~\eqref{eqn:triangle-Fl_I} gives a distinguished triangle
\begin{equation}
\label{eqn:triangle-kappa-2}
(q_I)^*\Delta^{\IW,I,\mix}_{w_\lambda}\{\ell(w_0^I)\}\to \cE^{\IW,\mix}_{w_\lambda}\to(q_I)^*(i_\lambda^{I})_*(i_\lambda^{I})^*\cE^{\IW,I,\mix}_{w_\lambda}\{\ell(w_0^I)\}\xrightarrow{[1]}.
\end{equation}

We now want to identify the triangles~\eqref{eqn:triangle-kappa} and~\eqref{eqn:triangle-kappa-2}.
Applying $\kappa^{-1}$ to the composition of the first arrow in~\eqref{eqn:triangle-kappa} and the second arrow in~\eqref{eqn:triangle-kappa-2} gives a map
\[
(j_\lambda)_!(j_\lambda)^*\tilt^{\mix}_{w_\lambda}\to\kappa^{-1}((q_I)^*(i_\lambda^{I})_*(i_\lambda^{I})^*\cE^{\IW,I,\mix}_{w_\lambda}\{\ell(w_0^I)\}).
\]
This map is $0$ by adjunction, since the second term is supported on
$X_\lambda \smallsetminus U_\lambda$. Hence our original composition also vanishes, and as in~\S\ref{ss:adverse} we obtain a unique morphism of triangles
\[
\begin{tikzcd}[column sep=small]
 \kappa((j_\lambda)_!(j_\lambda)^*\tilt^{\mix}_{w_\lambda}) \ar[r] \ar[d] & \cE_{w_\lambda}^{\IW,\mix} \ar[r] \ar[d] & \kappa((i_\lambda)_*(i_\lambda)^*\tilt^{\mix}_{w_\lambda}) \ar[r,"{[1]}"] \ar[d] & \ \\
 (q_I)^*\Delta^{\IW,I,\mix}_{w_\lambda}\{\ell(w_0^I)\} \ar[r] & \cE_{w_\lambda}^{\IW,\mix} \ar[r] & (q_I)^*(i_\lambda^{I})_*(i_\lambda^{I})^*\cE^{\IW,I,\mix}_{w_\lambda}\{\ell(w_0^I)\} \ar[r,"{[1]}"] & \ 
\end{tikzcd}
\]
whose second vertical arrow is the identity. Consider now the space of maps between the first object in~\eqref{eqn:triangle-kappa-2} and the third object in~\eqref{eqn:triangle-kappa}: we have
\begin{multline*}
\Hom_{\Dmix_{\mathcal{IW}}(\Fl, \bk)} \left( (q_I)^*\Delta^{\IW,I,\mix}_{w_\lambda}\{\ell(w_0^I)\}, \kappa((i_\lambda)_*(i_\lambda)^*\tilt^{\mix}_{w_\lambda}) \right) \\
\cong \Hom_{\Dmix_{\mathcal{IW}}( \Fl_I, \bk)} \left( \Delta^{\IW,I,\mix}_{w_\lambda}\{\ell(w_0^I)\}, (q_I)_*\kappa((i_\lambda)_*(i_\lambda)^*\tilt^{\mix}_{w_\lambda}) \right). 
\end{multline*}
By Lemma~\ref{lem:pushforward-D-N-Whit-mix},
the object $(q_I)_*\kappa((i_\lambda)_*(i_\lambda)^*\tilt^{\mix}_{w_\lambda})$ belongs to the triangulated subcategory of $\Dmix_{\mathcal{IW}}( \Fl_I, \bk)$ generated by the objects $\nabla^{\IW,I,\mix}_v$ with $v\neq w_\lambda$; hence this $\Hom$-space vanishes. There is thus a unique morphism of triangles 
from~\eqref{eqn:triangle-kappa-2} to ~\eqref{eqn:triangle-kappa} whose middle arrow is the identity
and, as in~\S\ref{ss:adverse}, applying~\cite[Proposition~1.1.9]{bbd} to both compositions of these morphisms gives that they are inverse isomorphisms. We have finally identified~\eqref{eqn:triangle-kappa} and~\eqref{eqn:triangle-kappa-2}, hence proved in particular that there exists an isomorphism 
\begin{equation}
\label{eqn:adverse-lem}
\kappa((j_\lambda)_!(j_\lambda)^*\tilt^{\mix}_{w_\lambda})\cong (q_I)^*\Delta^{\IW,I,\mix}_{w_\lambda}\{\ell(w_0^I)\}.
\end{equation}

We can finally conclude: Lemma~\ref{lem:Whit-perverse} guarantees that the right-hand side of~\eqref{eqn:adverse-lem} is perverse; hence $(j_\lambda)_!(j_\lambda)^*\tilt^{\mix}_{w_\lambda}$ is adverse, which finishes the proof.

\section{Application: a singular version
of the mixed Finkelberg--Mirkovi\'c conjecture}
\label{s:singular-fm}

In this section we assume that $\bk$ is an algebraic closure of $\F_\ell$.

\subsection{Whittaker perverse sheaves on \texorpdfstring{$\Gr$}{Gr}}

We continue with the notation of~\S\ref{ss:IW-Dmix}. We also denote by $\dot B_\F^{\vee,+}$ the Borel subgroup of $\dot G^\vee_\F$ which is opposite to $\dot B^\vee_\F$ with respect to $\dot T^\vee_\F$. (In other words, as the notation suggests, $\dot B_\F^{\vee,+}$ is the \emph{positive} Borel subgroup of $\dot G^\vee_\F$.) We will now work with the $\F$-version of the affine Grassmannian (in its usual incarnation in terms of cosets for right multiplication), defined by
\[
 \Gr := \dot{G}_\F^{\vee} \bigl(\F( \hspace{-1pt} (\mathrm{t}) \hspace{-1pt} ) \bigr) / \dot{G}_\F^{\vee} \bigl(\F[ \hspace{-1pt} [\mathrm{t}] \hspace{-1pt} ] \bigr).
\]
We have a natural embedding
\[
 \bX = \dot{T}_\F^{\vee} \bigl(\F( \hspace{-1pt} (\mathrm{t}) \hspace{-1pt} ) \bigr) / \dot{T}_\F^{\vee} \bigl(\F[ \hspace{-1pt} [\mathrm{t}] \hspace{-1pt} ] \bigr) \to \Gr,
\]
and we will denote by $L_\lambda$ the image of $\lambda \in \bX$.

Let again $I \subset S$ be a subset, and consider the associated parabolic subgroup
$\dot{P}^{\vee,+}_{\F,I} \subset \dot{G}^{\vee}$ containing $\dot{B}^{\vee,+}_\F$. Let also $\dot{L}^{\vee}_{\F,I}$ be the Levi factor of $\dot{P}^{\vee,+}_{\F,I}$ containing $\dot{T}_\F^\vee$, and set $\dot{U}^{\vee}_{\F,I} := \dot{U}_\F^{\vee} \cap \dot{L}^{\vee}_{\F,I}$, where $\dot{U}_\F^{\vee}$ is the unipotent radical of the (negative) Borel subgroup $\dot{B}_\F^{\vee} \subset \dot{G}_\F^{\vee}$. We set
\[
 Q_I := \mathrm{ev}^{-1} \bigl( (\dot{P}^{\vee,+}_{\F,I})_{\mathrm{u}} \cdot \dot{U}^{\vee}_{\F,I} \bigr),
\]
where $\mathrm{ev} : \dot{G}_\F^{\vee} \bigl(\F[ \hspace{-1pt} [\mathrm{t}] \hspace{-1pt} ] \bigr) \to \dot{G}_\F^{\vee}$ is the natural morphism, and $(\dot{P}^{\vee,+}_{\F,I})_{\mathrm{u}}$ is the unipotent radical of $\dot{P}^{\vee,+}_{\F,I}$. Fixing an Artin--Schreier local system on $\dot{U}^{\vee}_{\F,I}$ and pulling it back to $Q_I$, we obtain a notion of ``Whittaker'' complexes on $\Gr$ (see e.g.~\S\hyperref[ss:Whittaker-appendix]{A.B});\footnote{Here we understand the \'etale derived category of $\bk$-sheaves on $\Gr$ as the direct limit of the categories of $\bk_0$-sheaves, where $\bk_0$ runs over finite subfields of $\bk$.} we denote the corresponding category of parity complexes by $\Par_{\Whit,I}(\Gr,\bk)$, and the associated mixed derived category by $\Dmix_{\Whit,I}(\Gr,\bk)$. (In case $I=\varnothing$, this construction recovers the familiar Iwahori-constructible categories.) As in~\S\ref{ss:IW-Dmix}, this category is endowed with a natural ``perverse'' t-structure, whose heart will be denoted $\Perv^\mix_{\Whit,I}(\Gr,\bk)$, and the corresponding standard and costandard objects are perverse (by an appropriate analogue of Proposition~\ref{prop:D-N-perverse-Whit}). By standard arguments (see e.g.~\cite{modrap2}), this implies that the realization functor (see~\cite[\S 3.1]{bbd} or~\cite[Appendix]{be}) induces an equivalence of categories
\begin{equation}
 \label{eqn:realization-functor}
 \Db \Perv^\mix_{\Whit,I}(\Gr,\bk) \simto \Dmix_{\Whit,I}(\Gr,\bk).
\end{equation}

The orbits of $Q_I$ on $\Gr$ are parametrized in a natural way by $\bX$, where $\lambda$ corresponds to the orbit of $L_\lambda$. To understand the combinatorics of orbits on (partial) affine flag varieties, it is usually simpler to work with the \emph{negative} Iwahori subgroup $\Iw_\F$ (defined as $\mathrm{ev}^{-1}(\dot B_\F^{\vee})$). But here, because we will eventually want to combine our constructions with those of~\cite{ar:agsr} and~\cite[Section~11]{prinblock}, we instead work with the \emph{positive} Iwahori subgroup $\Iw_\F^+$, defined as $\mathrm{ev}^{-1}(\dot B_\F^{\vee,+})$. We now explain how to compare the resulting combinatorics of orbits.

We set
\[
 \Gr'_\F := \dot{G}_\F^{\vee} \bigl(\F[ \hspace{-1pt} [\mathrm{t}] \hspace{-1pt} ] \bigr) \backslash \dot{G}_\F^{\vee} \bigl(\F( \hspace{-1pt} (\mathrm{t}) \hspace{-1pt} ) \bigr) \quad
\text{and} \quad
 \Fl'_\F := \Iw_\F \backslash \dot{G}_\F^{\vee} \bigl(\F( \hspace{-1pt} (\mathrm{t}) \hspace{-1pt} ) \bigr).
\]
Then the $\Iw_\F$-orbits on $\Fl'_\F$ (for the action induced by right multiplication) are parametrized in a natural way by $\Waff$, and those on $\Gr'_\F$ by $\Waffmin$. Consider a ``Cartan'' anti-automorphism of $\dot G^\vee_\F$ which acts as the identity on $\dot T^\vee_\F$ and sends $\dot{B}_\F^{\vee}$ to $\dot{B}_\F^{\vee,+}$. (With the notation of~\cite[Remark~11.3(2)]{prinblock}, this antiautomorphism can be chosen as the composition of the $\F$-version of the automorphism $\varphi$ with the map $g \mapsto g^{-1}$.) This map induces an isomorphism $\Gr'_\F \to \Gr$, which sends the $\Iw_\F$-orbit corresponding to $w_\lambda$ to the $\Iw^+_\F$-orbit of~$L_\lambda$.

\begin{lem}
\label{lem:QI-orbit}
 For $\lambda \in \bX$, the $Q_I$-orbit on $\Gr$ labelled by $\lambda$ supports a nonzero local system which is $Q_I$-equivariant against the pullback of the Artin--Schreier local system iff $\lambda \in -\bXpp_I$.
\end{lem}

\begin{proof}
 We translate our problem in terms of $\Gr'_\F$ following the principles presented above. We will denote by $Q_I^-$ the analogue of $Q_I$ where the roles of positive and negative roots are switched (so that $Q_I^-$ is the appropriate analogue in the present setting of the subgroup $\UKM^K \cdot \UKM_K^\#$ of~\S\hyperref[ss:Whittaker-appendix]{A.B}). Then we have to show that the $Q_I^-$-orbit labelled with $w_\lambda$ supports a (nonzero) Whittaker local system iff $\lambda \in -\bXpp_I$. By the general considerations in~\S\hyperref[ss:Whittaker-appendix]{A.B}, the latter condition holds iff $v w_\lambda$ is minimal in $v w_\lambda W_I$ for any $v \in W$.
 
 First, we assume that $w_\lambda$ satisfies this property. Then $\ell(vw_\lambda w)=\ell(vw_\lambda) + \ell(w)=\ell(v) + \ell(w_\lambda) + \ell(w)$ for any $v \in W$ and $w \in W_I$. In particular, $w_\lambda w_0^I$ belongs to $\Waffmin$; it must then coincide with $w_{w_0^I(\lambda)}$. In view of~\cite[Lemma~10.2]{prinblock}, this implies that $w_0^I(\lambda)$ belongs to $\bXpp_I$, so that $\lambda \in -\bXpp_I$.
 
 Conversely, assume that $\lambda \in -\bXpp_I$. Then, by the converse implication in~\cite[Lemma~10.2]{prinblock}, $w_{w_0^I(\lambda)} w_0^I$ belongs to $\Waffmin$, hence coincides with $w_\lambda$; moreover we have $\ell(vw_\lambda w)=\ell(v) +\ell(w_\lambda w)=\ell(v) + \ell(w_\lambda) + \ell(w)$ for any $v \in W$ and $w \in W_I$. This implies that $v w_\lambda$ is minimal in $v w_\lambda W_I$ for any $v \in W$, and finishes the proof.
 \qed
\end{proof}

For $\lambda \in -\bXpp_I$,
we will denote by $\cE^{\Whit,I}(\lambda)$, $\cJ_!^{\Whit,I}(\lambda)$, $\cJ_*^{\Whit,I}(\lambda)$, $\cJ_{!*}^{\Whit,I}(\lambda)$, $\cT^{\Whit,I}(\lambda)$ the corresponding normalized indecomposable parity complex, standard mixed perverse sheaf, costandard mixed perverse sheaf, simple mixed perverse sheaf and indecomposable mixed tilting perverse sheaf associated with $\lambda$ respectively; see~\cite{modrap2} for details on these notions. We will also denote by $\cE^{\Whit,I,\mix}(\lambda)$ the object of $\Dmix_{\Whit,I}(\Gr,\bk)$ consisting of the complex with $\cE^{\Whit,I}(\lambda)$ in degree $0$, and $0$ in other degrees. (When $I=\varnothing$, we will sometimes omit the superscripts.)

\subsection{Averaging functor}

We have a natural ``averaging'' functor
\[
 \Av_I : \Db_{\Whit, \varnothing}(\Gr,\bk) \to \Db_{\Whit,I}(\Gr,\bk),
\]
defined by the same procedure as in~\S\hyperref[ss:appendix-averaging]{A.C}.

In the following lemma we use the notion of the ``naive'' quotient of an additive category by a full additive subcategory as in~\S\hyperref[ss:appendix-averaging]{A.C}.

\begin{lem}
\label{lem:AvI-parity}
 The functor $\Av_I$ sends parity complexes to parity complexes. Moreover, it induces an equivalence of additive categories
 \[
  \Par_{\Whit, \varnothing}(\Gr,\bk) 
  \sslash
  \langle \cE(\lambda) : \lambda \notin -\bXpp_I \rangle_{\oplus, \Z} \simto \Par_{\Whit,I}(\Gr,\bk),
 \]
 where $\langle \cE(\lambda) : \lambda \notin -\bXpp_I \rangle_{\oplus, \Z}$ is the full subcategory of $\Par_{\Whit,\varnothing}(\Gr,\bk)$ whose objects are direct sums of shifts of objects of the form $\cE(\lambda)$ with $\lambda \notin -\bXpp_I$.
\end{lem}

\begin{proof}
In view of Lemma~\ref{lem:QI-orbit}, this statement is the analogue in the present context of Proposition~\ref{prop:equivalence-parity}. 
\qed
\end{proof}

From Lemma~\ref{lem:AvI-parity} we obtain in particular that $\Av_I$ induces a functor
\[
 \Av^{\mix}_I : \Dmix_{\Whit, \varnothing}(\Gr,\bk) \to \Dmix_{\Whit,I}(\Gr,\bk)
\]
on bounded homotopy categories.

\begin{lem}
\label{lem:Av-D-N}
 For $\lambda \in -\bXpp_I$ we have
 \[
  \Av^{\mix}_I(\cJ_!(\lambda)) \cong \cJ^{\Whit,I}_!(\lambda), \quad \Av^{\mix}_I(\cJ_*(\lambda)) \cong \cJ^{\Whit,I}_*(\lambda).
 \]
\end{lem}

\begin{proof}
These claims are special cases of the analogue in the present context of Lemma~\ref{lem:average-mix-K}. In particular, from the proof of Lemma~\ref{lem:QI-orbit}, for $\lambda\in -\bXpp_I$, $w_\lambda$ is minimal in $w_\lambda W_I$ and so the shifts appearing in the second statement of Lemma~\ref{lem:average-mix-K} are $0$.
\qed
\end{proof}

\subsection{Study of tilting objects in \texorpdfstring{$\ExCoh(\tcN_I)$}{ExCoh}}

We denote by $\Tilt(\ExCoh(\tcN_I))$ the full additive subcategory of $\ExCoh(\tcN_I)$ whose objects are the tilting objects. (This notion does make sense now that Corollary~\ref{cor:hw} is proved.) By the general theory of (graded) highest weight categories, we know that the isomorphism classes of indecomposable objects in this category are in a natural bijection with $\bXpp_I \times \Z$. For any $\lambda \in \bXpp_I$, we will denote by $\tilt_I^{\mathrm{exo}}(\lambda)$ the object associated with the pair $(\lambda,0)$; then the object associated with $(\lambda,n)$ is $\tilt_I^{\mathrm{exo}}(\lambda) \langle n \rangle$. It is also known that the ``realization'' functor
\begin{equation}
\label{eqn:KbTilt-equiv}
\Kb \Tilt(\ExCoh(\tcN_I)) \to \Db \Coh^{\dot G \times \Gm}(\tcN_I)
\end{equation}
provided by~\cite[Proposition~2.2]{amrw} is an equivalence of categories. (To construct this functor one needs to choose a ``filtered version'' of $\Db \Coh^{\dot G \times \Gm}(\tcN_I)$; here we take the filtered version constructed in~\cite[\S 3.1]{bbd}.)

\begin{lem}\phantomsection
\label{lem:Pi-tilt}
 \begin{enumerate}
  \item[\rm (1)]
  \phantomsection\label{it:Pi-Tilt-1}
  For any $\cT$ in $\Tilt(\ExCoh(\tcN_\varnothing))$, the object $\Pi_I(\cT)$ belongs to $\ExCoh(\tcN_I)$, and is tilting therein.
  \item[\rm (2)]
  \phantomsection\label{it:Pi-Tilt-2}
  If $\lambda \notin -\bXpp_I$ we have $\Pi_I(\tilt_\varnothing^{\mathrm{exo}}(\lambda))=0$.  
 \end{enumerate}
\end{lem}

\begin{proof}
 (\hyperref[it:Pi-Tilt-1]{1})
 Let $\cT\in\Tilt(\ExCoh(\tcN_\varnothing))$. Then $\cT$ is an extension of objects of the form $\Delta_\varnothing (\lambda)\langle n \rangle$ for $\lambda\in\bX$ and $n\in\Z$. By~\cite[Corollary~9.21 and Proposition~9.24(2)]{prinblock}, we deduce that $\Pi_I(\cT)$ is an extension (in the sense of triangulated categories) of objects of the form $\Delta_I(\mu)\langle n \rangle [m]$ for $\mu\in\bXpp_I$, $n\in\Z$ and $m\in\Z_{\leq 0}$, hence that
 \begin{equation}
 \label{eqn:PiI-tilt-1}
  \Hom(\Pi_I(\cT), \nabla_I(\nu)\langle n \rangle [m])=0 \quad \text{if $m>0$.}
 \end{equation}
Now, $\cT$ is also an extension of objects $\nabla_\varnothing(\lambda)\langle n \rangle$ for $\lambda\in\bX$ and $n\in\Z$. Using now~\cite[Corollary~9.21 and Proposition~9.24(1)]{prinblock}, we deduce similarly that 
\begin{equation}
 \label{eqn:PiI-tilt-2}
  \Hom(\Delta_I(\nu), \Pi_I(\cT)\langle n \rangle [m])=0 \quad \text{if $m>0$.}
 \end{equation}
By~\cite[Lemma~4]{bez:ctm}, the properties~\eqref{eqn:PiI-tilt-1} and~\eqref{eqn:PiI-tilt-2} imply that $\Pi_I(\cT)$ belongs to $\ExCoh(\tcN_I)$, and is tilting therein.
 
 (\hyperref[it:Pi-Tilt-2]{2})
 As observed in the proof of Lemma~\ref{lem:QI-orbit}, if $\lambda \notin -\bXpp_I$ then there exists $v \in W$ such that $vw_\lambda$ is not minimal in $vw_\lambda W_I$, or in other words such that $vw_\lambda$ admits a reduced expression ending with a simple reflection $s \in I$. In view of~\cite[Proof of Corollary~4.2]{mr:etspc} this shows that $\tilt_\varnothing^{\mathrm{exo}}(\lambda)$ is then a direct summand of an object which is killed by $\Pi_I$.
 \qed
\end{proof}

\subsection{Exotic sheaves and mixed perverse sheaves}

Recall now the equivalence $\Psi$ from~\S\ref{ss:mixed-Gr}. Here we will rather consider the variant of this equivalence considered in~\cite[\S 11.3]{prinblock}, which will be denoted
\[
 P : \Dmix_{\Whit, \varnothing}(\Gr,\bk) \simto \Db \Coh^{\dot G \times \Gm}(\tcN_\varnothing).
\]
Note that in~\cite{prinblock} the affine Grassmannian is defined over the complex numbers, while here we work with \'etale sheaves on the $\F$-version of this variety. The fact that these two constructions give rise to equivalent categories follows from the general principles from~\cite[\S 6.1]{bbd}.

We now denote by
\[
\Perv_{\sph}(\Gr,\bk)
\]
the abelian category of $\dot{G}_\F^{\vee} \bigl(\F[ \hspace{-1pt} [\mathrm{t}] \hspace{-1pt} ] \bigr)$-equivariant $\bk$-perverse sheaves on $\Gr$. This category is equipped with a symmetric monoidal structure given by the convolution product $\star$; moreover if we denote by $\Rep^{\textbf{f}}(\dot G)$ the category of finite-dimensional algebraic $\dot G$-modules, then the \emph{geometric Satake equivalence} provides an equivalence of abelian monoidal categories
\[
 \Sat : \bigl( \Perv_{\sph}(\Gr,\bk), \star \bigr) \to \bigl( \Rep^{\textbf{f}}(\dot G), \otimes \bigr);
\]
see~\cite{mv:gld} for the original source and~\cite{br} for a more detailed exposition of the proof. The same considerations as in~\cite[\S 11.2]{prinblock} (see also~\cite[\S 4.4]{bgmrr}) show that the convolution construction also provides an action of the monoidal category  $( \Perv_{\sph}(\Gr,\bk), \star )$ on $\Dmix_{\Whit,I}(\Gr,\bk)$ on the right, which will also be denoted $\star$.

The following theorem is a ``parabolic version'' of the main results of~\cite{ar:agsr} and~\cite{mr:etsps}.

\begin{thm}\label{thm:parabolic-abg}
 There exists an equivalence of triangulated categories
 \[
  P_I : \Dmix_{\Whit,I}(\Gr,\bk) \simto \Db \Coh^{\dot G \times \Gm}(\tcN_I)
 \]
such that 
\begin{enumerate}
 \item[\rm (1)]
 \phantomsection\label{it:thm-PI-shifts}
 there exists an isomorphism of functors $P_I \circ \langle 1 \rangle \cong \langle 1 \rangle [1] \circ P_I$;
 \item[\rm (2)]
 \phantomsection\label{it:thm-PI-D-N}
 for any $\lambda \in -\bXpp_I$ there exist isomorphisms
 \begin{gather*}
  P_I(\cJ_!^{\Whit,I}(\lambda)) \cong \Delta_I(w_0^I (\lambda)), \quad P_I(\cJ_*^{\Whit,I}(\lambda)) \cong \nabla_I(w_0^I (\lambda)), \\
  P_I(\cE^{\Whit,I,\mix}(\lambda)) \cong \tilt^{\mathrm{exo}}_I(w_0^I (\lambda));
 \end{gather*}
 \item[\rm (3)]
 \phantomsection\label{it:thm-PI-Sat}
 for any $\cF$ in $\Dmix_{\Whit,I}(\Gr,\bk)$ and $\cG \in \Perv_{\sph}(\Gr,\bk)$, there exists a bifunctorial isomorphism
 $$P_I(\cF \star \cG) \cong P_I(\cF) \otimes \Sat(\cG);$$
 \item[\rm (4)]
 \phantomsection\label{lem:thm-PI-diag}
 the following diagram commutes up to isomorphism:
\[
\begin{tikzcd}[column sep=large]
  \Dmix_{\Whit, \varnothing}(\Gr,\bk) \ar[r, "P", "\sim"'] \ar[d, "\Av_I^\mix"'] & \Db \Coh^{\dot G \times \Gm}(\tcN_\varnothing) \ar[d, "\Pi_I"] \\
  \Dmix_{\Whit,I}(\Gr,\bk) \ar[r, "P_I", "\sim"'] & \Db \Coh^{\dot G \times \Gm}(\tcN_I).
\end{tikzcd}
\]
\end{enumerate}
\end{thm}

\begin{proof}
 The equivalence $P$ restricts to an equivalence of additive categories
 \begin{equation*}
  P^{\prime}: \Par_{\Whit,\varnothing}(\Gr,\bk) \simto \Tilt(\ExCoh(\tcN_\varnothing))
 \end{equation*}
 sending $\cE(\lambda)$ to $\tilt^{\mathrm{exo}}_\varnothing(\lambda)$, see~\cite[Proposition~8.4]{ar:agsr}.
 Lemma~\ref{lem:AvI-parity} and Lemma~\ref{lem:Pi-tilt} imply that $\Pi_I \circ P^{\prime}$ factors through a functor
 \begin{equation*}
 P^{\prime}_I : \Par_{\Whit,I}(\Gr,\bk) \to \Tilt(\ExCoh(\tcN_I)).
\end{equation*}
Passing to bounded homotopy categories and composing with the equivalence~\eqref{eqn:KbTilt-equiv} we deduce a functor
\begin{equation*}
 P_I : \Dmix_{\Whit,I}(\Gr,\bk) \to \Db\Coh^{\dot G \times \Gm}(\tcN_I)
\end{equation*}
which satisfies~(\hyperref[it:thm-PI-shifts]{1}). 

Consider now the diagram
{\small
\[
\begin{tikzcd}[column sep=small]
\Dmix_{\Whit,\varnothing}(\Gr,\bk) \ar[d, "\Av^\mix_I"'] \ar[r, equal] \ar[rrr, bend left=15, "P"]&
\Kb\Par_{(\Iw)}(\Gr,\bk)  \ar[d, "\Kb\Av_I"] \ar[r, "\Kb P^{\prime}", "\sim"'] &
  \Kb\Tilt(\ExCoh(\tcN_{\varnothing})) \ar[d, "\Kb\Pi_{I}"] \ar[r, "\mathrm{real}", "\sim"'] & 
  \Db\Coh^{\dot G \times \Gm}(\tcN_{\varnothing}) \ar[d, "\Pi_I"]  
  \\
\Dmix_{\Whit,I}(\Gr,\bk)  \ar[r, equal] \ar[rrr, bend right=15, "P_I"'] &
\Kb\Par_{\Whit,I}(\Gr,\bk) \ar[r, "\Kb P_{I}^{\prime}"] & \Kb\Tilt(\ExCoh(\tcN_{I})) \ar[r, "\mathrm{real}", "\sim"'] &
  \Db\Coh^{\dot G \times \Gm}(\tcN_I) 
\end{tikzcd}
\]
}where the maps labelled ``real'' are 
the functors~\eqref{eqn:KbTilt-equiv}.
The leftmost square in this diagram commutes by definition, and the middle square commutes by construction of $P_I'$. The rightmost square commutes by~\cite[Proposition~2.3]{amrw}. (To be able to apply this result we need to check that the functor $\Pi_I$ admits a ``lift'' to the filtered versions. This however follows from~\cite[Example~A.2]{be}.)
The bottom part of the diagram commutes by definition of $P_I$. 

We claim that the top part of the diagram also commutes (up to isomorphism). This will again follow from~\cite[Proposition~2.3]{amrw} once we justify that the functor $P$ lifts to filtered versions. (Here, the filtered version of $\Dmix_{\Whit,\varnothing}(\Gr,\bk)$ that we consider is the same as in~\cite[Comments preceding Lemma~2.4]{amrw}; the corresponding realization functor is then the identity.)
In fact, $P$ is defined (see~\cite[\S 7]{ar:agsr}) by applying an additive functor 
\[
P^{0}:\Par_{\Whit,\varnothing}(\Gr,\bk)\to\Coh^{\dot G \times \Gm}(\tcN_{\varnothing})
\] 
termwise to complexes of parity complexes. Using the filtered version of the category $\Dmix_{\Whit,\varnothing}(\Gr,\bk)$ constructed in~\cite[\S 2.5]{ar:frobenius} and that of $\Db\Coh^{\dot G \times \Gm}(\tcN_{\varnothing})$ constructed in~\cite[\S 3.1]{bbd}, we can again apply $P^{0}$ termwise to filtered objects formed from $\Par_{\Whit,\varnothing}(\Gr,\bk)$ to obtain filtered objects formed from $\Coh^{\dot G \times \Gm}(\tcN_{\varnothing})$, and hence obtain the required lift. 
This finishes the justification of the commutativity of the diagram in~(\hyperref[lem:thm-PI-diag]{4}).

For $\lambda \in -\bXpp_I$, by~\cite[Proposition~9.24]{prinblock} we have
 \[
  \Pi_I(\Delta_\varnothing(\lambda)) \cong \Delta_I(w_0^I (\lambda)), \quad \Pi_I(\nabla_\varnothing(\lambda)) \cong \nabla_I(w_0^I (\lambda)).
 \]
 Comparing this with Lemma~\ref{lem:Av-D-N} we deduce that 
 \[
 P_I(\cJ_!^{\Whit,I}(\lambda)) \cong \Delta_I(w_0^I (\lambda)) \quad \text{and} \quad P_I(\cJ_*^{\Whit,I}(\lambda)) \cong \nabla_I(w_0^I (\lambda)).
 \]
 Then standard arguments (see e.g.~\cite[Theorem~6.5]{amrw}) show that $P_I$ is an equivalence of categories. In particular, this functor must send indecomposable objects to indecomposable objects, and the third isomorphism in (\hyperref[it:thm-PI-D-N]{2}) follows.
 
To conclude, it only remains to prove~(\hyperref[it:thm-PI-Sat]{3}). 
By construction of the convolution action of $\Perv_{\sph}(\Gr,\bk)$ on $\Dmix_{\Whit,I}(\Gr,\bk)$, it suffices to construct such an isomorphism when $\cG$ is parity (in addition to being perverse), i.e.~when $\Sat(\cG)$ is a tilting $\dot G$-module. In this case the functor $(-) \otimes \Sat(\cG)$ stabilizes $\Tilt(\ExCoh(\tcN_\varnothing))$ by~\cite[Proposition~4.10]{mr:etspc}.
The equivalence $P$ intertwines the functors $(-) \star \cG$ and $(-) \otimes \Sat(\cG)$ (see~\cite[Proposition~7.2]{ar:agsr}); therefore the same property holds for its restriction $P^{\prime}$. 

The functor $\Pi_I$ clearly commutes with the functors $(-) \otimes \Sat(\cG)$; from this we deduce that $(-) \otimes \Sat(\cG)$ also preserves the subcategory $\Tilt(\ExCoh(\tcN_I))$. Now the functor $\Av_I$ commutes with $(-) \star \cG$ (see e.g.~\cite[(11.1.1)]{rw} for a similar statement); hence the latter functor preserves the kernel of the former, namely the subcategory $\langle \cE(\lambda) : \lambda \notin -\bXpp_I \rangle_{\oplus, \Z}$ of $\Par_{\Whit, \varnothing}(\Gr,\bk)$. By construction of the functor $P_I^{\prime}$ out of $P'$, $\Av_I$ and $\Pi_I$, we finally deduce that this functor also intertwines the functors $(-) \star \cG$ and $(-) \otimes \Sat(\cG)$. And using~\cite[Proposition~2.3]{amrw} once again we deduce~(\hyperref[it:thm-PI-Sat]{3}).
\qed
\end{proof}

\subsection{Relation with representations of reductive groups}
\label{ss:relation}

From now on we fix a connected reductive group $G$ over $\bk$ with simply-connected derived subgroup, and assume that $\ell>h$, where $h$ is the Coxeter number of $G$.
We also fix a maximal torus and a Borel subgroup $T \subset B \subset G$. We then assume that $\dot G$, resp.~$\dot B$, resp.~$\dot T$, is the Frobenius twist of $G$, resp.~$B$, resp.~$T$. (Of course, $\dot G$ is a reductive group that is isomorphic to $G$, but it plays a different conceptual role.) Note that $\ell$ is automatically very good for $\dot G$, so that the assumptions of~\S\ref{ss:notation} hold.

We will identify the lattice of characters of $T$ with $\bX$, in such a way that the composition of the Frobenius morphism $T \to \dot T$ with the character $\lambda \in \bX=X^*(\dot T)$ is the character $\ell \lambda$ of $T$. We will consider the ``dilated and shifted'' action of $\Waff$ on $\bX$ defined by
\[
 w \cdot_\ell \mu = w(\mu+\rho)-\rho, \quad t_\lambda \cdot_\ell \mu = \mu + \ell\lambda
\]
for $w \in W$ and $\lambda,\mu \in \bX$. (Here, $\rho \in \frac{1}{2} \bX$ is as usual the halfsum of the positive roots.)
 
Denote by $\Rep_\varnothing(G)$ the ``extended principal block" of the category $\Rep^{\textbf{f}}(G)$ of finite-dimensional algebraic $G$-modules, that is, the Serre subcategory generated by the simple modules whose highest weight has the form $w \cdot_\ell 0$ with $w \in \Waff$. Here the weight $w \cdot_\ell 0$ is dominant iff $w$ belongs to the subset $\Waffmin \subset \Waff$. In particular, $\Rep_\varnothing(G)$ contains the Weyl and induced modules of highest weight $w \cdot_\ell 0$ for $w \in \Waffmin$, denoted $\weyl(w \cdot_\ell 0)$ and $\coweyl(w \cdot_\ell 0)$ respectively. The corresponding simple object will be denoted $\irr(w \cdot_\ell 0)$.

Given a subset $I\subset S$, we 
let $\Rep_I(G)$ be the Serre subcategory of $\Rep^{\textbf{f}}(G)$ generated by the simple modules whose highest weight has the form $w \cdot_\ell (-\varsigma_I)$ for $w\in \Waff$. This category is a direct summand of $\Rep^{\textbf{f}}(G)$ and is ``singular at $I$" in that the stabilizer of $-\varsigma_I$ under the dot action of $\Waff$ is the subgroup $W_I$ of $W$ generated by $I$. Here, the weight $w \cdot_{\ell} (-\varsigma_I)$ is dominant iff $w$ belongs to $\WaffminI \subset \Waff$ (see~\cite[\S 10.1]{prinblock}).
In particular, $\Rep_I(G)$ contains the Weyl and induced modules of highest weight $w \cdot_\ell (-\varsigma_I)$ for $w \in \WaffminI$, denoted $\weyl(w \cdot_\ell (-\varsigma_I))$ and $\coweyl(w \cdot_\ell (-\varsigma_I))$ respectively, and the corresponding simple module $\irr(w \cdot_\ell (-\varsigma_I))$.

Recall the bijection $\bX \simto \Waffmin$ considered in~\S\ref{ss:mixed-Gr}. As explained in~\cite[Lem\-ma~10.2]{prinblock}, this bijection restricts to a bijection $\bXpp_I \simto \WaffminI$.
Recall also the functor
\[
 \Phi_I : \Db \Coh^{\dot G \times \Gm}(\tcN_I) \to \Db \Rep_I(G)
\]
constructed in~\cite[\S 10.3]{prinblock}. (In the notation of~\cite{prinblock}, we have $\Phi_I := \Omega_I \circ \varkappa_I$.) According to~\cite[Proposition~10.6]{prinblock}, this is a \emph{degrading functor} with respect to $\langle 1 \rangle [1]$: that is, there exists a natural isomorphism $\Phi_I \circ \langle 1 \rangle [1] \simto \Phi_I$, such that $\Phi_I$ induces an isomorphism
\begin{equation}
\label{eqn:Phi_I-degrading}
\bigoplus_{n \in \Z} \Hom_{\Db \Coh^{\dot G \times \Gm}(\tcN_I)}(\cF, \cG\la n\ra[n]) \simto \Hom_{\Db \Rep_I(G)}(\Phi_I(\cF), \Phi_I(\cG))
\end{equation}
for all $\cF, \cG \in \Db \Coh^{\dot G \times \Gm}(\tcN_I)$. Moreover, by~\cite[Proposition~10.3]{prinblock} we have
\begin{equation}
\label{eqn:Phi_I-t-exact}
 \Phi_I(\Delta_I(\lambda)) \cong \weyl(w_\lambda \cdot_{\ell} (-\varsigma_I)), \quad \Phi_I(\nabla_I(\lambda)) \cong \coweyl(w_\lambda \cdot_{\ell} (-\varsigma_I)).
\end{equation}
It is clear from these properties that $\Phi_I$ is t-exact if $\Db \Coh^{\dot G \times \Gm}(\tcN_I)$ is endowed with the representation-theoretic t-structure, and $\Db \Rep_I(G)$ with its tautological t-structure. In particular, this provides a grading on the category $\Rep_I(G)$ in the sense of~\cite[Definition~4.3.1]{bgs} (see also~\cite[Definition~11.5]{prinblock}); in other words, under our present assumptions $\grRep(\tcN_I)$ is a ``graded version'' of $\Rep_I(G)$.

\begin{rmk}
\label{rmk:Pi^I-tilting}{\rm
 In view of~\cite[Theorem~8.16, Remark~8.17 and Proposition~9.25]{prinblock}, the functors $\Phi_\varnothing$ and $\Phi_I$ intertwine the ``geometric translation functors'' $\Pi_I$ and $\Pi^I$ and the usual translation functors for $G$-modules, denoted $T_\varnothing^I$ and $T_I^\varnothing$ in~\cite{prinblock}. Using~\cite[Proposition~E.11]{jantzen}, it follows that in the setting of Lemma~\ref{lem:Pi-TilRT} we in fact have $\Pi^I(\TilRT_I(\lambda)) \cong \TilRT_\varnothing(\lambda)$ under the present assumptions.}
\end{rmk}

\subsection{The singular Mirkovi\'c--Vilonen conjecture}

The following theorem is a ``singular analogue'' of~\cite[Proposition~11.6 and Theorem~11.7]{prinblock}. Here we denote by $\For^{\dot G}_G : \Rep^{\textbf{f}}(\dot G) \to \Rep^{\textbf{f}}(G)$ the restriction functor associated with the Frobenius morphism $G \to \dot G$.

\begin{thm}
\label{thm:singular-fm}
 \begin{enumerate}
  \item[\rm (1)]
  \phantomsection\label{it:product-perv}
  For any $\cF$ in $\Perv^{\mix}_{\Whit,I}(\Gr,\bk)$ and $\cG$ in $\Perv_{\sph}(\Gr,\bk)$, the object $\cF \star \cG$ belongs to $\Perv^{\mix}_{\Whit,I}(\Gr,\bk)$.
  \item[\rm (2)]
  There exists an exact functor
  \[
   \bQ_I : \Perv^{\mix}_{\Whit,I}(\Gr,\bk) \to \Rep_I(G)
  \]
 together with an isomorphism $\varepsilon_I : \bQ_I \simto \bQ_I \circ \langle 1 \rangle$ such that the triple $(\Perv^{\mix}_{\Whit,I}(\Gr,\bk),\bQ_I, \varepsilon_I)$ is a grading on $\Rep_I(G)$. In addition,
 \begin{enumerate}
  \item[\rm (a)] for any $\lambda \in -\bXpp_I$ we have
  \begin{align*}
   \bQ_I(\cJ_!^{\Whit,I}(\lambda)) \cong \weyl(w_{w_0^I(\lambda)} \cdot_{\ell} (-\varsigma_I)), &\quad \bQ_I(\cJ_*^{\Whit,I}(\lambda)) \cong \coweyl(w_{w_0^I(\lambda)} \cdot_{\ell} (-\varsigma_I)), \\
   \bQ_I(\cJ_{!*}^{\Whit,I}(\lambda)) \cong \irr(w_{w_0^I(\lambda)} \cdot_{\ell} (-\varsigma_I)), &\quad \bQ_I(\cT^{\Whit,I}(\lambda)) \cong \tilt(w_{w_0^I(\lambda)} \cdot_{\ell} (-\varsigma_I));
  \end{align*}
  \item[\rm (b)] 
  \phantomsection\label{it:bifunctorial-perv}
  for any $\cF$ in $\Perv^{\mix}_{\Whit,I}(\Gr,\bk)$ and $\cG$ in $\Perv_{\sph}(\Gr,\bk)$, there exists a bifunctorial isomorphism $\bQ_I(\cF \star \cG) \cong \bQ_I(\cF) \otimes \For^{\dot G}_G(\Sat(\cG))$.
 \end{enumerate}
 \end{enumerate}
\end{thm}

\begin{proof}
 The proof is the same as for~\cite[Theorem~11.7]{prinblock} but for completeness we repeat it. 
 We first construct a functor $\bQ_I : \Dmix_{\Whit,I}(\Gr,\bk) \to \Db\Rep_I(G)$ as the composition $\Phi_I\circ P_I$. Using the isomorphism from Theorem~\ref{thm:parabolic-abg}(\hyperref[it:thm-PI-shifts]{1}) and the natural isomorphism $\Phi_I \circ \langle 1 \rangle [1] \simto \Phi_I$ we obtain an isomorphism $\varepsilon_I : \bQ_I \simto \bQ_I \circ \langle 1 \rangle$. 
By the first two isomorphisms in Theorem~\ref{thm:parabolic-abg}(\hyperref[it:thm-PI-D-N]{2}) and~\eqref{eqn:Phi_I-t-exact}, we also obtain isomorphisms
\begin{equation}
\label{eqn:Q_I-t-exact}
\bQ_I(\cJ_!^{\Whit,I}(\lambda)) \cong \weyl(w_{w_0^I(\lambda)} \cdot_{\ell} (-\varsigma_I)), \quad \bQ_I(\cJ_*^{\Whit,I}(\lambda)) \cong \coweyl(w_{w_0^I(\lambda)} \cdot_{\ell} (-\varsigma_I))
\end{equation}
for any $\lambda \in -\bXpp_I$.
In particular, this shows that the complexes $\bQ_I(\cJ_!^{\Whit,I}(\lambda))$ and $\bQ_I(\cJ_*^{\Whit,I}(\lambda))$ belong to $\Rep_I(G)$, which implies that $\bQ_I$ is t-exact if the category $\Dmix_{\Whit,I}(\Gr,\bk)$ is endowed with the perverse t-structure and $\Db\Rep_I(G)$ with its tautological t-structure. We will still denote by $\bQ_I$ the restriction of this functor to the hearts of these t-structures; this functor provides the wished-for exact functor $\Perv^{\mix}_{\Whit,I}(\Gr,\bk) \to \Rep_I(G)$. By~\eqref{eqn:realization-functor}, Theorem~\ref{thm:parabolic-abg}, and~\eqref{eqn:Phi_I-degrading}, this functor induces isomorphisms
\begin{equation}
 \label{eqn:Q_I-degrading}
 \bigoplus_{n \in \Z} \Ext^{k}_{\Perv^{\mix}_{\Whit,I}(\Gr,\bk)}(\cF, \cG\la n\ra) \simto \Ext^{k}_{\Rep_I(G)}(\bQ_I(\cF), \bQ_I(\cG))
 \end{equation}
 for any $\cF$ and $\cG$ in $\Perv^{\mix}_{\Whit,I}(\Gr,\bk)$ and any $k \in \Z$. In particular, this implies that $\bQ_I$ is faithful.
Now $\cJ_{!*}^{\Whit,I}(\lambda)$ is the image of any nonzero morphism $\cJ_!^{\Whit,I}(\lambda)\to\cJ_*^{\Whit,I}(\lambda)$ and $\irr(w_{w_0^I(\lambda)} \cdot_{\ell} (-\varsigma_I))$ is the image of any nonzero morphism $\weyl(w_{w_0^I(\lambda)} \cdot_{\ell} (-\varsigma_I))\to\coweyl(w_{w_0^I(\lambda)} \cdot_{\ell} (-\varsigma_I))$. Combining these facts with~\eqref{eqn:Q_I-t-exact} gives that $\bQ_I(\cJ_{!*}^{\Whit,I}(\lambda))\cong\irr(w_{w_0^I(\lambda)} \cdot_{\ell} (-\varsigma_I))$, hence finally that $(\Perv^{\mix}_{(\Whit,I)}(\Gr,\bk),\bQ_I, \varepsilon_I)$ gives a grading on $\Rep_I(G)$.

As $\bQ_I$ is exact and given the isomorphisms~\eqref{eqn:Q_I-t-exact}, one sees that, for any $\lambda \in -\bXpp_I$, $\bQ_I(\cT^{\Whit,I}(\lambda))$ is a tilting $G$-module, and that it admits $\tilt(w_{w_0^I(\lambda)} \cdot_{\ell} (-\varsigma_I))$ as a direct summand. The isomorphism~\eqref{eqn:Q_I-degrading} provides a ring isomorphism
\begin{equation}
\bigoplus_{n \in \Z} \Hom_{\Perv^{\mix}_{(\Whit,I)}(\Gr,\bk)} \Bigl( \cT^{\Whit,I}(\lambda), \cT^{\Whit,I}(\lambda)\la n\ra \Bigr) \simto \End_{\Rep_I(G)}(\bQ_I(\cT^{\Whit,I}(\lambda))).
\end{equation}
Since the left-hand side is local by~\cite[Theorem~3.1]{gg}, this shows that
$\bQ_I(\cT^{\Whit,I}(\lambda))$ is indecomposable, and so isomorphic to $\tilt(w_{w_0^I(\lambda)} \cdot_{\ell} (-\varsigma_I))$.

By Theorem~\ref{thm:parabolic-abg}(\hyperref[it:thm-PI-Sat]{3}), for $\cF$ in $\Dmix_{\Whit,I}(\Gr,\bk)$ and $\cG$ in $\Perv_{\sph}(\Gr,\bk)$ we have
\[
\bQ_I(\cF \star \cG) \cong \Phi_I(P_I(\cF) \otimes \Sat(\cG))
\]
Combining this with the isomorphisms in~\cite[Theorem~1.1 and Theorem~1.2]{prinblock}, we deduce
a bifunctorial isomorphism
\[
\bQ_I(\cF \star \cG) \cong \bQ_I(\cF) \otimes \For^{\dot G}_G(\Sat(\cG))
\]
in $\Db\Rep_I(G)$.
If $\cF$ is in $\Perv^{\mix}_{\Whit,I}(\Gr,\bk)$ then this implies that $\bQ_I(\cF \star \cG)$ belongs to $\Rep_I(G)$. Since $\bQ_I$ is t-exact and does not kill any nonzero object, this in turn implies that $\cF \star \cG$ is perverse, which proves Points~(\hyperref[it:product-perv]{1}) and~(\hyperref[it:bifunctorial-perv]{2b}), and finishes the proof.
\qed
\end{proof}

\appendix{ A}{Whittaker mixed perverse sheaves on partial flag varieties}
\label{sec:appendix}


\renewcommand{\theequation}{A.\arabic{equation}}
\renewcommand{\thelem}{A.\arabic{lem}}

In this appendix we assume that the reader is familiar (to a certain extent at least) with the theory of parity complexes (from~\cite{jmw}) and of mixed derived categories (from~\cite{modrap2}). Our aim is to study such objects and categories in the case of Whittaker sheaves on partial flag varieties of Kac--Moody groups. The case of Bruhat-constructible sheaves on partial flag varieties (corresponding, in the notation used below, to the case when $K=\varnothing$) is known, mainly from~\cite{jmw, modrap2}, as is that of Whittaker sheaves on the full flag variety (corresponding to the case $J=\varnothing$), mainly from~\cite{rw, amrw}. The general case will usually be deduced from one of these special cases.

\subsection{Notation}

In this section we consider the setting of~\cite[Part III]{rw} or~\cite[\S\S6.1--6.2]{amrw}. In particular, we consider an algebraically closed field $\F$ of characteristic $p>0$ and a Kac--Moody root datum $(I,\bX,\{\alpha_i\}_{i \in I}, \{\alpha^\vee_i\}_{i \in I})$. (Note that the symbol ``$I$'' used here is unrelated to the set $I$ considered in the body of the paper.) Let $\GKM$ be the associated Kac--Moody group over $\F$ in the sense of Mathieu. We also denote by $\BKM \subset \GKM$ the Borel subgroup, and by $\Flag:=\GKM/\BKM$ the associated flag variety. (See~\cite[\S 9.1]{rw} for a reminder on this construction, and for references to the original sources.) If $W$ is the Weyl group of $\GKM$, and if $S \subset W$ are the simple reflections (in canonical bijection with $I$), then we have a decomposition into $\BKM$-orbits
\[
 \Flag = \bigsqcup_{w \in W} \Flag_w,
\]
where $\Flag_w$ is a locally closed subvariety isomorphic to an affine space of dimension $\ell(w)$.

For any subset $J \subset I$ of finite type we also have a partial flag variety $\Flag^J$. We will denote by $W_J \subset W$ the (finite) subgroup generated by the simple reflections corresponding to elements in $J$, and by $W^J \subset W$ the subset of elements $w$ such that $w$ is minimal in $w W_J$. Then we have a stratification
\[
 \Flag^J = \bigsqcup_{w \in W^J} \Flag_w^J
\]
with $\Flag_w^J \cong \mathbb{A}^{\ell(w)}$.
We also have a natural proper morphism of ind-schemes $q_J : \Flag \to \Flag^J$. For any $w \in W^J$ and $v \in W_J$, we have $q_J(\Flag_{wv})=\Flag^J_w$, and the morphism $\Flag_{wv} \to \Flag_w^J$ induced by $q_J$ identifies with the natural projection from $\mathbb{A}^{\ell(w)+\ell(v)}$ to $\mathbb{A}^{\ell(w)}$. (Note however that it is not known---at least to us---if $q_J$ is a smooth morphism in general; see~\cite[Remark~9.2.1]{rw} for details on this question.)

\subsection{Whittaker derived categories}
\label{ss:Whittaker-appendix}

We let $\ell$ be a prime number different from $p$, and $\bk$ be either a finite field of characteristic $\ell$, or a finite extension of $\Ql$.
Then it makes sense to consider \'etale $\bk$-sheaves on $\Flag^J$ (for $J \subset I$ of finite type).
 We will assume that $\bk$ contains a primitive $p$-th root of unity; after fixing a choice of such a root of unity we can consider the associated Artin--Schreier local system $\AS$ on $\mathbb{G}_{\mathbf{a},\F}$.

Let now $K \subset I$ be another subset of finite type, and consider the associated parabolic subgroup $\PKM_K$ of $\GKM$, and its pro-unipotent radical $\UKM^K$. Let also $\LKM_K \subset \PKM_K$ be the Levi subgroup, and $\UKM_K^\# \subset \LKM_K$ be the unipotent radical of the Borel subgroup of $\LKM_K$ which is opposite to $\BKM \cap \LKM_K$ with respect to the canonical maximal torus. Then the orbits of $\UKM^K \cdot \UKM_K^\#$ on $\Flag^J$ are also in a natural bijection with $W^J$. We will denote by ${}^K \hspace{-1.5pt} \Flag^J_w$ the orbit associated with $w$. (When $J$ or $K$ is $\varnothing$, we will usually omit the corresponding superscript. This convention will be applied more generally to any notation used in this appendix and involving $J$ or $K$.)

After choosing an identification of each simple root subgroup of $\UKM_K^\#$ with the additive group $\mathbb{G}_{\mathbf{a},\F}$, the quotient $\UKM_K^\# / [\UKM_K^\#,\UKM_K^\#]$ identifies with a product of $\# K$ copies of $\mathbb{G}_{\mathbf{a},\F}$. Composing with the addition map to $\mathbb{G}_{\mathbf{a},\F}$, we obtain a group homomorphism $\UKM_K^\# / [\UKM_K^\#,\UKM_K^\#] \to \mathbb{G}_{\mathbf{a},\F}$. The composition of this morphism with the projection
\[
 \UKM^K \cdot \UKM_K^\# \twoheadrightarrow \UKM_K^\# \twoheadrightarrow \UKM_K^\# / [\UKM_K^\#,\UKM_K^\#] 
\]
will be denoted $\chi_K$.

We will denote by
\[
 \Db_{\Whit,K}(\Flag^J,\bk)
\]
the (\'etale) $(\UKM^K \cdot \UKM_K^\#, (\chi_K)^* \AS)$-equivariant derived category of $\bk$-sheaves on $\Flag^J$. (See~\cite[Appendix~A]{modrap1} for a brief review of the construction of this category. When $K=\varnothing$, one recovers the usual Bruhat-constructible derived category.) 

In the case $J=\varnothing$, as explained in~\cite[\S 11.1]{rw} (see also~\cite[Lemma~4.2.1]{by} for more details), the orbit ${}^K \hspace{-1.5pt} \Flag_w$ supports a nonzero $(\UKM^K \cdot \UKM_K^\#, (\chi_K)^* \AS)$-equivariant local system iff $w$ is minimal in $W_K w$. The subset of $W$ consisting of elements satisfying this condition will be denoted $\KW$. For the orbits on $\Flag^J$, we observe that, for $w \in W^J$, the orbit ${}^K \hspace{-1.5pt} \Flag_w^J$ supports a nonzero $(\UKM^K \cdot \UKM_K^\#, (\chi_K)^* \AS)$-equivariant local system iff each orbit in $(q_J)^{-1}({}^K \hspace{-1.5pt} \Flag_w^J)$ supports a $(\UKM^K \cdot \UKM_K^\#, (\chi_K)^* \AS)$-equivariant local system, i.e.~iff $w$ belongs to
\[
 \KW^J := \{w \in W^J \mid \forall v \in W_J, \, wv \in \KW\}.
\]
In fact, using standard combinatorics of Coxeter groups one can check that
\begin{equation}
\label{eqn:KWJ}
 \KW^J = \{w \in W^J \mid ww_0^J \in \KW\},
\end{equation}
where $w_0^J$ is the longest element in $W_J$.

For any element $w \in \KW^J$, we have standard and costandard perverse sheaves in $\Db_{\Whit,K}(\Flag^J,\bk)$, denoted $\KD^J_w$ and $\KN^J_w$ respectively, and obtained as $!$- and $*$-extensions of the perversely shifted rank-$1$ $(\UKM^K \cdot \UKM_K^\#, (\chi_K)^* \AS)$-equivariant local system on ${}^K \hspace{-1.5pt} \Flag_w^J$. (The fact that these objects are perverse sheaves is guaranteed by~\cite[Corollaire~4.1.3]{bbd}.)

The following lemma is an extension of~\cite[Lemma~11.1.1]{rw}, with an essentially identical proof.

\begin{lemA}
\label{lem:pushforward-D-N-Whit}
 Let $w \in \KW$, and write $w=w^J w_J$ with $w^J \in W^J$ and $w_J \in W_J$. (Then $w^J$ automatically belongs to $\KW$.)
 \begin{enumerate}
  \item[\rm (1)]
  \phantomsection\label{pushforward-D-N-Whit-1}
  If $w^J \notin \KW^J$, then we have
  \[
   (q_J)_* \KD_w = 0 = (q_J)_* \KN_w.
  \]
  \item[\rm (2)]
  \phantomsection\label{pushforward-D-N-Whit-2}
  If $w^J \in \KW^J$, then we have
  \[
   (q_J)_* \KD_w \cong \KD^J_{w^J} [-\ell(w_J)], \quad (q_J)_* \KN_w \cong \KN^J_{w^J} [\ell(w_J)].
  \]
 \end{enumerate}
\end{lemA}

\begin{proof}[Sketch of proof]
 We have $(q_J)({}^K \hspace{-1.5pt} \Flag_w) = {}^K \hspace{-1.5pt} \Flag_{w^J}^J$. By compatibility of the pushforward functors with composition, $(q_J)_* \KD_w$ and $(q_J)_* \KN_w$ are respectively the $*$- and $!$-pushforward of the object on ${}^K \hspace{-1.5pt} \Flag_{w^J}^J$ obtained as the $*$- and $!$-pushforward under the restriction of $q_J$ of the perversely shifted rank-$1$ $(\UKM^K \cdot \UKM_K^\#, (\chi_K)^* \AS)$-equivariant local system on ${}^K \hspace{-1.5pt} \Flag_w$. If $w^J \notin \KW^J$, then ${}^K \hspace{-1.5pt} \Flag_{w^J}^J$ does not support any nonzero $(\UKM^K \cdot \UKM_K^\#, (\chi_K)^* \AS)$-equivariant object, proving~(\hyperref[pushforward-D-N-Whit-1]{1}). If $w^J \in \KW^J$, then the map ${}^K \hspace{-1.5pt} \Flag_w \to {}^K \hspace{-1.5pt} \Flag_{w^J}^J$ induced by $q_J$ is a trivial fibration with fiber $\mathbb{A}^{\ell(w_J)}$, and the restriction of our local system to this fiber is trivial; this implies~(\hyperref[pushforward-D-N-Whit-2]{2}).
 \qed
\end{proof}

As in~\cite[\S 11.2]{rw}, we can consider the $*$-even, $*$-odd, $!$-even, $!$-odd, and parity complexes in $\Db_{\Whit,K}(\Flag^J,\bk)$. The same arguments as for~\cite[Lemma~11.2.1]{rw}, using Lemma~\ref{lem:pushforward-D-N-Whit} as a starting point instead of~\cite[Lemma~11.1.1]{rw}, implies the following.

\begin{propA}
\label{prop:pushforward-qJ}
 Let $\mathcal{F} \in \Db_{\Whit,K}(\Flag,\bk)$.
 \begin{enumerate}
  \item[\rm (1)] If $\mathcal{F}$ is $*$-even, then $(q_J)_* \mathcal{F}$ is $*$-even.
  \item[\rm (2)] If $\mathcal{F}$ is $!$-even, then $(q_J)_* \mathcal{F}$ is $!$-even.
  \item[\rm (3)] If $\mathcal{F}$ is parity, then $(q_J)_* \mathcal{F}$ is parity.
 \end{enumerate}
\end{propA}

The general theory of parity complexes from~\cite{jmw} guarantees that, for any $w \in \KW^J$, there exists at most one indecomposable parity complex which is supported on the closure of ${}^K \hspace{-1.5pt} \Flag_{w}^J$ and whose restriction to ${}^K \hspace{-1.5pt} \Flag_{w}^J$ is a perversely shifted rank-$1$ local system. Proposition~\ref{prop:pushforward-qJ} guarantees that such an object indeed exists: in fact by~\cite[Remark~11.2.4]{rw} there exists a parity complex in $\Db_{\Whit,K}(\Flag,\bk)$ supported on ${}^K \hspace{-1.5pt} \Flag_{w}$ and whose restriction to ${}^K \hspace{-1.5pt} \Flag_{w}$ is a perversely shifted rank-$1$ local system. The image of this object under $(q_J)_*$ then admits a direct summand with the appropriate properties. As in~\cite{jmw} we deduce that isomorphism classes of indecomposable parity complexes in $\Db_{\Whit,K}(\Flag^J,\bk)$ are parametrized (in the obvious way) by $\KW^J \times \Z$. These comments also show that any parity complex in $\Db_{\Whit,K}(\Flag^J,\bk)$ is a direct summand of an object of the form $(q_J)_* \mathcal{F}$ with $\mathcal{F}$ parity in $\Db_{\Whit,K}(\Flag,\bk)$.

\begin{propA}
 \label{prop:pullback-qJ}
 Let $\mathcal{F} \in \Db_{\Whit,K}(\Flag^J,\bk)$.
  If $\mathcal{F}$ is parity, then $(q_J)^* \mathcal{F}$ and $(q_J)^! \mathcal{F}$ are parity.
\end{propA}

\begin{proof}[Sketch of proof]
The comments before the statement of the proposition show that it suffices to prove a similar statement for the functors $(q_J)^* (q_J)_*$ and $(q_J)^! (q_J)_*$. Now it follows from Proposition~\ref{prop:pushforward-qJ} and the definitions that the first, resp.~second, of these functors sends $*$-even, resp.~$!$-even, complexes to $*$-even, resp.~$!$-even, complexes (and similarly for odd). But the same considerations as in~\cite[Lemma~9.4.2]{rw} show that these functors differ only by a cohomological shift; hence they send parity complexes to parity complexes.
\qed
\end{proof}

\subsection{Averaging functor}
\label{ss:appendix-averaging}

We have a natural ``averaging'' functor
\[
 \Av_K^J : \Db_{\Whit,\varnothing}(\Flag^J,\bk) \to \Db_{\Whit,K}(\Flag^J,\bk) 
\]
which can be defined as in~\cite[\S A.2]{modrap1}. More precisely, a priori there exist two versions of this functor: a $*$-version $\Av_{K,*}^J$ (defined in terms of a $*$-pushforward) and a $!$-version $\Av_{K,!}^J$ (defined in terms of a $!$-pushforward). However there exists a canonical morphism $\Av_{K,!}^J \to \Av_{K,*}^J$. The composition of each of these functors with the forgetful functor from the $\BKM$-equivariant derived category to $\Db_{\Whit,\varnothing}(\Flag^J,\bk)$ identifies with the convolution product with the object $\KD_{\id}^\varnothing = \KN_{\id}^{\varnothing}$; therefore our morphism is an isomorphism on objects in the essential image of this functor. Since this essential image generates $\Db_{\Whit,\varnothing}(\Flag^J,\bk)$ as a triangulated category, the 5-lemma then implies that the morphism $\Av_{K,!}^J \to \Av_{K,*}^J$ is an isomorphism. Our notation $\Av_K^J$ stands for either of these isomorphic functors.

\begin{lemA}
\label{lem:Av-parity}
 The functor $\Av_K^J$ sends parity complexes to parity complexes.
\end{lemA}

\begin{proof}
 The case $J=\varnothing$ is treated in~\cite[Corollary~11.2.3]{rw}. The general case follows, using the facts that
 \[
  (q_J)_* \circ \Av_K \cong \Av_K^J \circ (q_J)_*,
 \]
that $(q_J)_*$ sends parity complexes to parity complexes (see Proposition~\ref{prop:pushforward-qJ}) and that any parity complex in $\Db_{\Whit,K}(\Flag^J,\bk)$ is a direct summand of an object of the form $(q_J)_* \mathcal{F}$ with $\mathcal{F}$ parity in $\Db_{\Whit,K}(\Flag,\bk)$ (see the comments before Proposition~\ref{prop:pullback-qJ}).
\qed
\end{proof}

We now denote by $\Par_{\Whit,K}(\Flag^J,\bk)$ the full subcategory of $\Db_{\Whit,K}(\Flag^J,\bk)$ whose objects are parity. Lemma~\ref{lem:Av-parity} shows that $\Av_K^J$ restricts to a functor $\Par_{\Whit,\varnothing}(\Flag^J,\bk) \to \Par_{\Whit,K}(\Flag^J,\bk)$, which will also be denoted $\Av_K^J$. For $w \in W^J$, we will denote by $\cE^J_w$ the indecomposable object of $\Par_{\Whit,\varnothing}(\Flag^J,\bk)$ parametrized by $w$ (see the comments following Proposition~\ref{prop:pushforward-qJ}). We will also denote by
\[
\langle \cE^J_w : w \in W^J \smallsetminus \KW^J \rangle_{\oplus,\Z}
\]
the full subcategory of $\Par_{\Whit,\varnothing}(\Flag^J,\bk)$ whose objects are the direct sums of cohomological shifts of objects of the form $\cE^J_w$ with $w \in W^J \smallsetminus \KW^J$. Then we consider the ``naive'' quotient
\[
 \Par_{\Whit,\varnothing}(\Flag^J,\bk) \sslash
 \langle \cE^J_w : w \in W^J \smallsetminus \KW^J \rangle_{\oplus,\Z},
\]
i.e.~the additive category whose objects are those of $\Par_{\Whit,\varnothing}(\Flag^J,\bk)$, and whose morphisms are obtained from those in $\Par_{\Whit,\varnothing}(\Flag^J,\bk)$ by quotienting by the morphisms which factor through an object of $\langle \cE^J_w : w \in W^J \smallsetminus \KW^J \rangle_{\oplus,\Z}$.

\begin{lemA}
\label{lem:pullback-parity}
 For any $w \in W^J$ we have
 \[
  (q_J)^* \cE_w^J[\ell(w_0^J)] \cong \cE_{w w_0^J}.
 \]
\end{lemA}

\begin{proof}
The proof is copied from~\cite[Proposition~3.5]{williamson-IC}. By Proposition~\ref{prop:pullback-qJ}, the complex $(q_J)^* \cE_w^J[\ell(w_0^J)]$ is parity. The orbit $\Flag_{ww_0^J}$ is open in the support of this object, and its restriction to this stratum is $\underline{\bk}[\ell(ww_0^J)]$. Hence
\[
 (q_J)^* \cE_w^J[\ell(w_0^J)] \cong \cE_{w w_0^J} \oplus \cG
\]
for some parity complex $\cG$ in $\Db_{\Whit,\varnothing}(\Flag,\bk)$. This isomorphism also shows that the restriction of $\cE_{w w_0^J}$ to any stratum $\Flag_x$ with $x \in wW_J$ is $\underline{\bk}[\ell(ww_0^J)]$; it follows (using distinguished triangles associated with open/closed decompositions) that the restriction of $(q_J)_* \cE_{w w_0^J}$ to $\Flag^J_w$ is $\bigoplus_{z \in W_J} \underline{\bk}[\ell(ww_0^J) - 2 \ell(z)]$. Since this stratum is open in the support of this object, we deduce that
\[
 (q_J)_* \cE_{w w_0^J} \cong \bigoplus_{z \in W_J} \cE_w^J [\ell(w_0^J) - 2 \ell(z)] \oplus \cG'
\]
for some parity complex $\cG'$ in $\Db_{\Whit,\varnothing}(\Flag^J,\bk)$, and then that
\[
 (q_J)_* (q_J)^* \cE_w^J[\ell(w_0^J)] \cong \bigoplus_{z \in W_J} \cE_w^J [\ell(w_0^J) - 2 \ell(z)] \oplus \cG' \oplus (q_J)_* \cG.
\]
On the other hand, by the projection formula we have
\[
 (q_J)_* (q_J)^* \cE_w^J[\ell(w_0^J)] \cong \bigoplus_{z \in W_J} \cE_w^J [\ell(w_0^J) - 2 \ell(z)],
\]
proving that $\cG' = (q_J)_* \cG =0$. From this one can deduce that $\cG=0$, which completes the proof.
\qed
\end{proof}

\begin{propA}
\label{prop:equivalence-parity}
Assume that $\mathrm{char}(\ell) \neq 2$.\footnote{As should be clear from the proof, this assumption can be refined to the one considered in~\cite[Theorem~11.5.1]{rw}.}
 The functor $\Av^J_K$ vanishes on $\langle \cE^J_w : w \in W^J \smallsetminus \KW^J \rangle_{\oplus,\Z}$, and induces an equivalence of categories
 \[
  \Par_{\Whit,\varnothing}(\Flag^J,\bk) 
  \sslash
  \langle \cE^J_w : w \in W^J \smallsetminus \KW^J \rangle_{\oplus,\Z} \simto \Par_{\Whit,K}(\Flag^J,\bk).
 \]
\end{propA}

\begin{proof}
 The case $J=\varnothing$ is equivalent to~\cite[Theorem~11.5.1]{rw}. Let us now explain how the general case can be deduced from this one. Let $w \in W^J \smallsetminus \KW^J$. By Lemma~\ref{lem:pullback-parity} we have $(q_J)^* \cE^J_w \cong \cE_{ww_0^J}$. Here $w w_0^J \notin \KW$ by~\eqref{eqn:KWJ}, so that $\Av_K(\cE_{ww_0^J})=0$. We deduce that
 \[
  \Av_K((q_J)^* \cE^J_w) \cong (q_J)^* \Av_K^J(\cE_w^J)=0,
 \]
which implies that $\Av_K^J(\cE_w^J)$ vanishes. This proves the first claim of the statement, and hence also that $\Av_K^J$ factors through a functor
\[
 \Par_{\Whit,\varnothing}(\Flag^J,\bk) 
 \sslash
 \langle \cE^J_w : w \in W^J \smallsetminus \KW^J \rangle_{\oplus,\Z} \to \Par_{\Whit,K}(\Flag^J,\bk).
\]
We will now argue that this functor is fully faithful; essential surjectivity is then easy to see.

We need to show that for any $\cF,\cG$ in $\Par_{\Whit,\varnothing}(\Flag^J,\bk)$, the functor $\Av_K^J$ induces an isomorphism between the quotient of $\Hom(\cF,\cG)$ by the morphisms factoring through an object of the subcategory $\langle \cE^J_w : w \in W^J \smallsetminus \KW^J \rangle_{\oplus,\Z}$ and $\Hom(\Av_K^J(\cF), \Av_K^J(\cG))$. The comments after Proposition~\ref{prop:pushforward-qJ} (in the special case $K=\varnothing$) show that we can assume that $\cG=(q_J)_* \cH$ for some $\cH$ in $\Par_{\Whit,\varnothing}(\Flag,\bk)$; then we have $\Av_K^J(\cG) = (q_J)_* \Av_K(\cH)$, and using adjunction and the obvious isomorphism $(q_J)^* \circ \Av_K^J \cong \Av_K \circ (q_J)^*$ we obtain isomorphisms
\begin{align*}
 \Hom(\cF,\cG) &\cong \Hom((q_J)^* \cF, \cH), \\
 \Hom(\Av_K^J(\cF), \Av_K^J(\cG)) &\cong \Hom(\Av_K((q_J)^* \cF), \Av_K(\cH)).
\end{align*}
Moreover, under these identifications our morphism is induced by $\Av_K$. Taking into account the known case $J=\varnothing$, it therefore suffices to prove that the isomorphism
\begin{equation}
\label{eqn:Hom-adjunction}
 \Hom(\cF,\cG) \simto \Hom((q_J)^* \cF, \cH)
\end{equation}
identifies the subspace $V_1$ of the left-hand side consisting of morphisms factoring through an object of $\langle \cE^J_w : w \in W^J \smallsetminus \KW^J \rangle_{\oplus,\Z}$ with the subspace $V_2$ of the right-hand side consisting of morphisms factoring through an object of $\langle \cE_w : w \in W \smallsetminus \KW \rangle_{\oplus,\Z}$.

From Lemma~\ref{lem:pullback-parity} and~\eqref{eqn:KWJ} it is clear that~\eqref{eqn:Hom-adjunction} maps $V_1$ into $V_2$. On the other hand, if $w \notin \KW$ then $\cE_w$ is annihilated by $\Av_K$, so the same is true for $(q_J)^* (q_J)_* \cE_w$. Hence $(q_J)_* \cE_w$ cannot admit as direct summands objects of the form $\cE_x^J[n]$ with $x \in \KW^J$; in other words this object belongs to $\langle \cE^J_w : w \in W^J \smallsetminus \KW^J \rangle_{\oplus,\Z}$.  It follows that the inverse map of~\eqref{eqn:Hom-adjunction} sends $V_2$ into $V_1$, as desired.
\qed
\end{proof}

\begin{rmkA}{\rm
It follows in particular from Proposition~\ref{prop:equivalence-parity} that the functor $\Av_K^J$ sends indecomposable parity complexes to indecomposable parity complexes.}
\end{rmkA}

\subsection{Mixed derived category}
\label{ss:appendix-Dmix}

Following~\cite{modrap2}, we define the ``mixed derived category''
\[
 \Dmix_{\Whit,K}(\Flag^J,\bk) := \Kb \Par_{\Whit,K}(\Flag^J,\bk).
\]
As in~\cite{modrap2} the autoequivalence induced by the cohomological shift in the category $\Par_{\Whit,K}(\Flag^J,\bk)$ will be denoted by $\{1\}$, and the cohomological shift (of complexes of objects of $\Par_{\Whit,K}(\Flag^J,\bk)$) will be denoted by $[1]$. This category also admits a ``Tate twist'' autoequivalence $\langle 1 \rangle$ defined as $\{-1\}[1]$.

The ``recollement'' formalism constructed in~\cite{modrap2} applies in this setting (see also~\cite[\S 6.2]{amrw}), and we have ``mixed'' standard and costandard objects
\[
 \KD^{J,\mix}_w \quad \text{and} \quad \KN^{J,\mix}_w.
\]
The category $\Dmix_{\Whit,K}(\Flag^J,\bk)$ also admits a natural ``perverse'' t-structure, which can be characterized in terms of these objects as in~\cite[Proposition~3.4]{modrap2}: we have
\begin{align*}
 {}^p \hspace{-1pt} \Dmix_{\Whit,K}(\Flag^J,\bk)^{\leq 0} &= \langle \KD_w^{J,\mix} \langle n \rangle [m] : w \in \KW^J, \, n \in \Z, \, m \in \Z_{\geq 0} \rangle_{\mathrm{ext}},\\
 {}^p \hspace{-1pt} \Dmix_{\Whit,K}(\Flag^J,\bk)^{\geq 0} &= \langle \KN_w^{J,\mix} \langle n \rangle [m] : w \in \KW^J, \, n \in \Z, \, m \in \Z_{\leq 0} \rangle_{\mathrm{ext}},
\end{align*}
where $\langle - \rangle_{\mathrm{ext}}$ means the subcategory generated under extensions by the given objects.

The functors $(q_J)_*$, $(q_J)^*$, $(q_J)^!$ and $\Av_K^J$ send parity complexes to parity complexes by Proposition~\ref{prop:pushforward-qJ}, Proposition~\ref{prop:pullback-qJ} and Lemma~\ref{lem:Av-parity} respectively. Therefore they induce functors between the corresponding mixed derived categories, which will be denoted by the same symbol.

We can now prove the ``mixed analogue'' of Lemma~\ref{lem:pushforward-D-N-Whit}.

\begin{lemA}\phantomsection
\label{lem:pushforward-D-N-Whit-mix}
 Let $w \in \KW$, and write $w=w^J w_J$ with $w^J \in W^J$ and $w_J \in W_J$. (Then $w^J$ automatically belongs to $\KW$.)
 \begin{enumerate}
  \item[\rm (1)]
  If $w^J \notin \KW^J$, then we have
  \[
   (q_J)_* \KD^{\mix}_w = 0 = (q_J)_* \KN^{\mix}_w.
  \]
  \item[\rm (2)]
  If $w^J \in \KW^J$, then we have
  \[
   (q_J)_* \KD^{\mix}_w \cong \KD^{J,\mix}_{w^J} \{ -\ell(w_J) \}, \quad (q_J)_* \KN^{\mix}_w \cong \KN^{J,\mix}_{w^J} \{ \ell(w_J) \}.
  \]
 \end{enumerate}
\end{lemA}

\begin{proof}
 The proof is identical to that of Lemma~\ref{lem:pushforward-D-N-Whit}, once we have proved the appropriate compatibility statement, in mixed derived categories, for the functor $(q_J)_*$ and the $*$- or $!$-pushforward functor under the embedding of a statum in $\Flag$ and $\Flag^J$. However, by adjunction it suffices to prove a similar statement for pullback functors. In turn, this property is clear from the Whittaker analogue of~\cite[Remark~2.7]{modrap2}.
 \qed
\end{proof}

\subsection{Standard and costandard objects}

The goal of this subsection is to prove the following claim.

\begin{propA}
\label{prop:D-N-perverse-Whit}
For any $w \in \KW^J$,
 the objects $\KD^{J,\mix}_w$ and $\KN^{J,\mix}_w$ belong to the heart of the perverse t-structure.
\end{propA}

For the proof of this claim we need some preliminary lemmas.

\begin{lemA}
\label{lem:average-mix}
 Let $w \in W$, and write $w=w_K w^K$ with $w_K \in W_K$ and $w^K \in \KW$. Then we have
 \[
  \Av_K(\Delta^{\mix}_w) \cong \KD^{\mix}_{w^K} \langle -\ell(w_K) \rangle, \quad \Av_K(\nabla^{\mix}_w) \cong \KN^{\mix}_{w^K} \langle \ell(w_K) \rangle.
 \]
\end{lemA}

\begin{proof}
 In the case $w_K=\id$, these isomorphisms are proved in~\cite[Lemma~6.1]{amrw}. We deduce the general case as follows. We will only give the details for the first isomorphism; the second one can be treated similarly. Recall from~\cite[Lemma~4.9]{modrap2} that there exists a morphism
 \[
  \Delta^{\mix}_\id \langle -\ell(w_K) \rangle \to \Delta^\mix_{w_K}
 \]
whose cone is an extension (in the sense of triangulated categories) of objects which belong to the essential image of the forgetful functors from some equivariant mixed derived categories for some parabolic subgroups of the form $\PKM_L$ with $\varnothing \neq L \subset K$. Convolving with $\Delta^{\mix}_{w^K}$ on the right and using~\cite[Proposition~4.4(1)]{modrap2} we deduce a morphism
\[
  \Delta^{\mix}_{w^K} \langle -\ell(w_K) \rangle \to \Delta^\mix_{w}
 \]
 whose cone satisfies a similar property. It is easily seen that this cone is killed by $\Av_K$; see e.g.~\cite[Lemma~4.4.6]{by} for similar considerations. Therefore we obtain an isomorphism
 \[
  \Av_K(\Delta^{\mix}_{w^K}) \langle -\ell(w_K) \rangle \simto \Av_K(\Delta^\mix_{w}),
 \]
 which concludes the proof.
 \qed
\end{proof}

\begin{lemA}
\label{lem:average-mix-K}
 Let $w \in W^J$. If $W_K w \cap \KW^J=\varnothing$, then
 \[
  \Av_K^J(\Delta^{J,\mix}_w)=0 \quad \text{and} \quad \Av_K^J(\nabla^{J,\mix}_w)=0.
 \]
Otherwise, write $w=w_K w^K$ with $w_K \in W_K$ and $w^K \in \KW$ (so that $w^K$ belongs to $\KW^J$). Then we have
 \[
  \Av_K^J(\Delta^{J,\mix}_w) \cong \KD^{J,\mix}_{w^K} \langle -\ell(w_K) \rangle, \quad \Av_K^J(\nabla^{J,\mix}_w) \cong \KN^{J,\mix}_{w^K} \langle \ell(w_K) \rangle.
 \]
\end{lemA}

\begin{proof}
 As in the non-mixed setting, the functor $(q_J)_*$ commutes with averaging functors. Hence using Lemma~\ref{lem:pushforward-D-N-Whit-mix} (in case $K=\varnothing$) we have
 \[
  \Av_K^J(\Delta^{J,\mix}_w) \cong \Av_K^J \circ (q_J)_*(\Delta^{\mix}_w) \cong (q_J)_* \circ \Av_K(\Delta^{\mix}_w).
 \]
Now we write $w=w_K w^K$ with $w_K \in W_K$ and $w^K \in \KW$. Applying Lemma~\ref{lem:average-mix} we deduce that
\[
 \Av_K^J(\Delta^{J,\mix}_w) \cong (q_J)_* ( \KD^{\mix}_{w^K} \langle -\ell(w_K) \rangle).
\]
Here $w^K$ automatically belongs to $W^J$. If $W_K w \cap \KW^J=\varnothing$ then $w^K \notin \KW^J$, so that $(q_J)_* ( \KD^{\mix}_{w^K})=0$ by Lemma~\ref{lem:pushforward-D-N-Whit-mix} (applied now with our choice of $K$). Otherwise we have $w^K \in \KW^J$, and $(q_J)_* ( \KD^{\mix}_{w^K})=\KD^{J,\mix}_{w^K}$ again by Lemma~\ref{lem:pushforward-D-N-Whit-mix}. The claims for costandard objects can be proved in a similar way.
 \qed
\end{proof}

\begin{proof}[Proof of Proposition~{\rm \ref{prop:D-N-perverse-Whit}}]
The special case $K=\varnothing$ is treated in~\cite[Theorem~4.7]{modrap2}. As written in~\cite{modrap2} this proof seems to use the fact that $q_J$ is smooth; to convince the reader that in fact this property does not play any crucial role, let us recall its main steps for the objects $\Delta^{J,\mix}_w$ (the other case is similar). This proof uses the $\BKM$-equivariant mixed derived category $\Dmix_\BKM(\Flag,\bk)$ defined as above but starting from the equivariant parity complexes; see~\cite[\S 3.5]{modrap2}. As above we let $w_0^J$ be the longest element in $W_J$; then $\overline{\Flag_{w_0^J}}$ is a smooth closed subvariety of $\Flag$, so that the shifted constant sheaf $\underline{\bk}_{\overline{\Flag_{w_0^J}}} \{\ell(w_0^J)\}$ defines an object of $\Dmix_\BKM(\Flag,\bk)$; see~\cite[Lemma~3.7]{modrap2}. We remark that
\[
 (q_J)^! \Delta^{J,\mix}_w \{-\ell(w_0^J)\} \cong (q_J)^! (q_J)_* \Delta^{\mix}_w \{-\ell(w_0^J)\} \cong \Delta^{\mix}_w \star^{\BKM} \underline{\bk}_{\overline{\Flag_{w_0^J}}} \{\ell(w_0^J)\}
\]
where the convolution product $\star^{\BKM}$ is as in~\cite[\S 4.3]{modrap2}. Here the first isomorphism uses Lemma~\ref{lem:pushforward-D-N-Whit-mix} (in the special case $K=\varnothing$), and the second one uses~\cite[Lemma~4.3]{modrap2} (see also~\cite[Lemma~9.4.2]{rw}). By~\cite[Lemma~3.7]{modrap2}, $\underline{\bk}_{\overline{\Flag_{w_0^J}}} \{\ell(w_0^J)\}$ is perverse. Using then~\cite[Proposition~4.6(2)]{modrap2}, we deduce that $(q_J)^! \Delta^{J,\mix}_w \{-\ell(w_0^J)\}$ 
belongs to ${}^p \hspace{-1pt} \Dmix_{\Whit,\varnothing}(\Flag,\bk)^{\geq 0}$.
Considering the $!$-pullback of this object to strata $\Flag_{x w_0^J}$ with $x \in W^J$, we deduce that $\Delta^{J,\mix}_w$ 
belongs to ${}^p \hspace{-1pt} \Dmix_{\Whit,\varnothing}(\Flag^J,\bk)^{\geq 0}$.
Since, by definition of the perverse t-structure, this object also belongs to ${}^p \hspace{-1pt} \Dmix_{\Whit,\varnothing}(\Flag^J,\bk)^{\leq 0}$, we finally conclude that it is perverse.

Now, let us deduce the case of a general subset $K$ (of finite type).
 In view of the characterization of the perverse t-structure in terms of standard and costandard objects (see~\S\hyperref[ss:appendix-Dmix]{A.D}), Lemma~\ref{lem:average-mix-K} implies that the functor $\Av_J^K$ is t-exact for the perverse t-structures. Since the objects $\Delta^{J,\mix}_w$ and $\nabla^{J,\mix}_w$ are known to be perverse, we deduce that $\Av_K^J(\Delta^{J,\mix}_w) \cong \KD^{J,\mix}_{w}$ and $\Av_K^J(\nabla^{J,\mix}_w) \cong \KN^{J,\mix}_{w}$ are perverse too (where the isomorphisms follow from Lemma~\ref{lem:average-mix-K} again).
  \qed
\end{proof}

Let us also note the following property.

\begin{lemA}
\label{lem:pullback-perverse}
 For any $w \in \KW^J$ we have isomorphisms
 \begin{align*}
  (q_J)^! \left( \KD^{J,\mix}_w \{-\ell(w_0^J)\} \right) &\cong (q_J)^* \left( \KD^{J,\mix}_w \{\ell(w_0^J)\} \right), \\ 
  (q_J)^! \left( \KN^{J,\mix}_w \{-\ell(w_0^J)\} \right) &\cong (q_J)^* \left( \KN^{J,\mix}_w \{\ell(w_0^J)\} \right).
 \end{align*}
Moreover, these objects are perverse.
\end{lemA}

\begin{proof}
As usual we only treat the case of standard objects; the case of costandard objects is similar.

We begin with the case $K=\varnothing$. Here we have
\[
 \Delta^{J,\mix}_w = (q_J)_* \Delta^{\mix}_w
\]
by Lemma~\ref{lem:pushforward-D-N-Whit-mix}, so the isomorphism between our two objects follows from the comparison of the two isomorphisms in~\cite[Lemma~9.4.2(1)]{rw}. We have already observed in the course of the proof of Proposition~\ref{prop:D-N-perverse-Whit} that these objects belong to ${}^p \hspace{-1pt} \Dmix_{\Whit,\varnothing}(\Flag,\bk)^{\geq 0}$. Now 
it is clear that the $*$-pullback of $(q_J)^* \left( \Delta^{J,\mix}_w \{\ell(w_0^J)\} \right)$ to a stratum $\Flag_x$ vanishes unless $x \in w W_J$ and that in this case it is isomorphic to $\underline{\bk}\{\ell(w)+\ell(w_0^J)\}$.
Since $\ell(w)+\ell(w_0^J) \geq \ell(x)$, we deduce that this object also belongs to ${}^p \hspace{-1pt} \Dmix_{\Whit,\varnothing}(\Flag,\bk)^{\leq 0}$.

 The case of a general subset $K$ follows from the case $K=\varnothing$ by applying the functor $\Av^J_K$, as in the proofs of Lemma~\ref{lem:average-mix-K} and Proposition~\ref{prop:D-N-perverse-Whit}.
  \qed
\end{proof}

\providecommand{\bysame}{\leavevmode\hbox to3em{\hrulefill}\thinspace}
%
%

\bibliographystyle{amsalpha}
\bibliographymark{References}
\def\cprime{$'$}

\end{document}